\documentclass[11pt, a4paper]{article} 

\usepackage[dvipsnames]{xcolor}
\usepackage{amsfonts}
\usepackage{amssymb}
\usepackage{amsmath}
\usepackage{amsthm}
\usepackage{latexsym}
\usepackage{graphicx}
\usepackage{epstopdf}
\usepackage[retainorgcmds]{IEEEtrantools}
\usepackage{pdfpages}
\usepackage{enumerate}
\usepackage{caption}
\usepackage{subcaption}	
\usepackage{amsthm}
\usepackage{setspace}
\usepackage{dsfont}
\usepackage{bbm}
\usepackage{hyperref}
\usepackage{datetime}
\usepackage[colorinlistoftodos, textsize=footnotesize, backgroundcolor=GreenYellow!25, linecolor=GreenYellow!25, bordercolor=GreenYellow!60]{todonotes}
\usepackage[percent]{overpic}
\usepackage{pict2e}
\presetkeys{todonotes}{fancyline}{}

\makeatletter

\newdimen\mainfontsize \mainfontsize=1\@ptsize pt
\topskip=\mainfontsize
   \@maxdepth=\maxdepth

\makeatother

\makeatletter
\renewcommand\@makefntext[1]{%
  \noindent\makebox[0.5em][r]{\@makefnmark}#1}
\makeatother

\makeatletter
\def\blfootnote{\xdef\@thefnmark{}\@footnotetext}
\makeatother

\newtheorem{theorem}{Theorem}[section]

\newtheorem{proposition}[theorem]{Proposition}

\theoremstyle{remark}

\theoremstyle{definition}

\newtheorem{example}[theorem]{Example}
\newtheorem{remark}[theorem]{Remark}
\newtheorem{assumption}[theorem]{Assumption}

\newcommand{\ud}{\,\mathrm{d}}
\newcommand{\vd}{\mathrm{d}}
\newcommand{\R}{\mathbb{R}}
\newcommand{\F}{\mathcal{F}}
\newcommand{\E}{\mathbb E}
\newcommand{\N}{\mathbb N}

\newcommand{\ep}{\epsilon}

\newcommand{\lt}{\left}
\newcommand{\rt}{\right}
\newcommand{\pt}{\partial}

\newcommand{\tend}{\rightarrow}

\def\P{{\mathbb P}}

\newcommand{\Ind}{\mathbbm{1}}
\newcommand{\supp}{\mathop{\mathrm{supp}}\nolimits}

\makeatletter
\@addtoreset{equation}{section}

\makeatother

\title{Optimal stopping of a Brownian bridge with an unknown pinning point}       
\author{Erik Ekstr\"om \footnotemark[1] \and 
Juozas Vaicenavicius
\footnote{Department of Mathematics, Uppsala University, Box 480, 751 06 Uppsala, Sweden 
\mbox{(\href{mailto:ekstrom@math.uu.se}{\nolinkurl{ekstrom@math.uu.se}}, \href{mailto:juozas.vaicenavicius@math.uu.se}{\nolinkurl{juozas.vaicenavicius@math.uu.se}}}).} }

\date{}
\begin{document}
\maketitle

\begin{abstract}
The problem of stopping a Brownian bridge with an unknown pinning point to maximise the expected value at the stopping time is studied. A few general properties, such as continuity and various bounds of the value function, are established. 
However, structural properties of the optimal stopping region are shown to crucially depend on the prior, and we provide a 
general condition for a one-sided stopping region. Moreover, a detailed analysis is conducted in the cases of the two-point and the mixed Gaussian priors, revealing a rich structure present in the problem.
 


%
\end{abstract}


\section{Introduction}

The Brownian bridge is a fundamental process in statistics and probability theory. For example, it appears in the limit for a normalised difference between the empirical and the true distribution, and it also plays a crucial role in the Kolmogorov-Smirnov test. Moreover, the Brownian bridge is a large population limit of the cumulative sum process obtained by sampling randomly from a finite population without replacement (see \cite{bR64}).

The notion of `pinning' refers to a situation in which the process is strongly attracted to 
a particular value at a certain time point. Besides the statistical examples mentioned above, pinning phenomena have also been observed in financial markets, see \cite{AL}; 
in particular, a tendency for stock prices to end up in the vicinity of a strike price at an option's maturity was reported. 
In \cite{AKL}, \cite{AL} and \cite{JIS}, efforts to explain the phenomenon based on models of
price impact are provided. 
Another natural example in which pinning with an unknown pinning point may occur is in connection with a presidential election or a referendum where the financial 
market favours one of the competing options. As the election date approaches, the market price of a financial asset is affected by information collected sequentially (opinion polls, new actions by the opposing parties, other news, etc.) 
Finally, processes exhibiting pinning naturally arise in the Kyle-Back model of insider trading, see \cite{B}, \cite{CD} and the references therein.

In the current article, we study an optimal stopping problem where the underlying payoff process is a Brownian bridge with an unknown pinning point. In a Bayesian formulation of the problem, the initial beliefs (knowledge) about the unknown pinning point are described by a prior distribution. As time evolves, the information obtained from observing the process is being used to update the initial beliefs about the pinning point. To accommodate arbitrary beliefs, in the general set-up of our problem, we allow for a general prior distribution. In the case of a known pinning point, an optimal stopping strategy is determined in \cite{lS}; also, see \cite{EW} for an alternative proof.  


In addition to a financial motivation (optimal liquidation), a statistical interpretation of the optimal stopping problem for a Brownian bridge with an unknown pinning point can also be given. An analogue of  Donsker's invariance principle in the case of randomly sampling from a finite population without replacement says that a normalised cumulative sum process converges in distribution to a Brownian bridge as the population size increases (see \cite[Theorem 13.1]{bR64}). Thus the Brownian bridge can be used as an approximation for the normalised cumulative sum process. If the value of each sample drawn is regarded as gain (loss if negative), then the cumulative sum represents the total gain from the samples drawn thus far. Hence our stopping problem  corresponds to the problem of terminating sequential sampling without replacement from a large finite population to maximise the expected gain when the exact mean of the population is unknown. 

Formulating the stopping problem under a general prior and using
filtering theory, the original problem can be rewritten as a Markovian optimal stopping problem, and we show that the value function is continuous and solves a free-boundary problem.  
Further general properties, however, appear to be scarce due to the sensitivity of the problem to the prior. 
Indeed, by general optimal stopping theory, solving a Markovian optimal stopping problem boils down to determining the so-called {\em continuation region} and its complement, 
the {\em stopping region}, and our study shows that different priors lead to structurally different optimal stopping strategies, suggesting that general structural properties are rare. In fact, for general prior distributions, multiple boundaries and stop-loss regions may exist. 

These structural complications make studies of the stopping problem under a general prior distribution infeasible and motivate us to focus on properties that help to understand certain classes of subproblems. In particular, for compactly supported 
prior distributions, we provide 
estimates for the value function based on suboptimal strategies and coupling arguments. These estimates may be used for
classifying points as belonging to the stopping region or to the continuation region. This is illustrated for the two-point distribution,
where, at least for some parameter values, large portions of the state space can be classified as stopping or continuation points.
Even in this simple case, the resulting optimal stopping problem has a rich structure, with the optimal stopping strategy often having a stop-loss level and multiple too-good-to-persist levels at a given fixed time. Also, the value function of the optimal stopping problem
cannot be extended to a continuous function at the pinning time. These features are results of the fact that the support of the two-point distribution is disconnected, and are expected to extend to other prior distributions with disconnected support.

Due to the complicated geometry of the stopping region, it is of interest to identify cases in which the problem has a simpler structure. We provide an analytic condition for a single  boundary separating the stopping region from the continuation region. Naturally, this structural property simplifies the analysis, and well-established optimal stopping theory may be applied, for example, to derive 
regularity 
properties of the boundary and to characterise the stopping boundary in terms of an integral equation. The analytic condition for a single stopping boundary is fulfilled by a large class of mixed Gaussian prior distributions, including the normal one.



\section{Model description and preliminary considerations}
\label{S:model}

Let $T>0$ be a deterministic time and $Z'=\{ Z'_t\}_{t\in [0,T]}$ be a one-dimensional Brownian motion with variance $\mathrm{Var}( Z_t')=\sigma^2t$, started at a point $Z'_0=\tilde{z}$. Also, let $\tilde{X}$ be a random variable with probability distribution $\tilde{\mu}$. If we condition  $Z'$ to satisfy $Z'_T=\tilde{X}$ a.s., we obtain a new process $\tilde{Z}$, which is a Brownian bridge pinning at a random point $\tilde{X}$ at time $T$. It is well-known that the Brownian bridge $\tilde{Z}$ admits a representation 
\begin{IEEEeqnarray}{rCl} \label{E:ZG}
\lt\{
\begin{array}{ll}
\vd \tilde{Z}_{t} = \frac{\tilde X-\tilde{Z}_{t}}{T-t} \ud t + \sigma \vd \tilde W_{t}, \; 0 \leq t < T,\\
\tilde{Z}_{0}=\tilde{z},
\end{array} \rt.
\end{IEEEeqnarray} 
as a stochastic differential equation (SDE), where $\tilde W$ is a standard Brownian motion.
The optimal stopping problem we are interested in is
\begin{IEEEeqnarray}{rCl} \label{E:OSO}
\tilde{V} &=& \sup_{0\leq\tau \leq T} \E[\tilde{Z}_{\tau}],
\end{IEEEeqnarray}
where the supremum is taken over random times $\tau$ that are $\mathcal F^{\tilde Z}$-stopping times (we use the convention 
that $\mathcal F^U=\{\mathcal F^U_t\}_{t\geq 0}$ is the filtration generated by a process $U$).
In particular, no a priori knowledge about the pinning point $\tilde X$, apart from its distribution, is assumed.

Define a process $Z_s := \frac{1}{\sigma \sqrt{T}} \tilde{Z}_{Ts}$, and note that $Z$ is a standard Brownian bridge pinning at 
$X:= \frac{1}{\sigma \sqrt{T}} \tilde{X}$ at time 1, i.e.~a standard Brownian motion conditioned to pin at $X$ at time $1$.
Here $X$ has distribution 
$\mu(\cdot)=\tilde\mu(\,\cdot \,\sigma\sqrt T)$,
and
the process $Z$ admits the SDE representation
\begin{IEEEeqnarray}{rCl} \label{E:ZN}
\lt \{
\begin{array}{ll}
\vd Z_s = \frac{ X - Z_s}{1-s} \ud s + \ud W_s\,, \;  0 \leq s < 1, \\
Z_0 = z,
\end{array} \rt.
\end{IEEEeqnarray}
where $z:= \frac{\tilde{z}}{\sigma \sqrt{T}}$ and $W_s := \frac{1}{\sqrt{T}}\tilde{W}_{Ts}$ is a Brownian motion. 
Moreover, 
$\tilde V=\sigma\sqrt T V$, where
\begin{IEEEeqnarray}{rCl} \label{E:OSR}
V &=& \sup_{0\leq\tau \leq 1} \E[Z_{\tau}]
\end{IEEEeqnarray}
and $\tau$ is a stopping time with respect to the filtration generated by the process $Z$. Thus, without loss of generality, 
we may consider the problem \eqref{E:OSR} instead of \eqref{E:OSO}.
Moreover, since an additive shift in the initial condition $z$ corresponds to an additive shift of the prior 
distribution $\mu$ and the value $V$,
we may without loss of generality assume that $Z_0=0$. Finally, we will throughout the paper 
assume that the distribution $\mu$ has a finite first moment.

\subsection{Review of the classical Brownian bridge}
In this subsection we briefly review the results for the special case of stopping a 
Brownian bridge with a known pinning point. 
Thus, suppose $\{Z_t\}_{t\in[0,1]}$ is a Brownian bridge as in \eqref{E:ZN} pinning at a deterministic point $r \in \R$, i.e.~$X=r$. 
The stopping problem \eqref{E:OSR} can be embedded into a Markovian framework by 
defining a Markovian value function 
\begin{IEEEeqnarray*}{rCl} 
v_r(t,z) = \sup_{\tau \in \mathcal{T}^Z_{1-t}} \E\lt[ Z^{t,z}_{t+\tau} \rt], \quad (t,z) \in [0,1) \times \R,
\end{IEEEeqnarray*}
where $\mathcal{T}^Z_{1-t}$ denotes the set of stopping times with respect to the process $Z=Z^{t,z}$,
where the indices indicate that $Z_t^{t,z}=z$.
The value function $v_r$ has an explicit solution (derived in \cite{lS}, see also \cite[Section 2]{EW})
\begin{IEEEeqnarray}{rCl}\label{E:known}
v_r(t,z) = 
\lt \{
\begin{array}{ll}
r+ \sqrt{2\pi(1-t)} (1-\beta^2)\exp\lt( \frac{(z-r)^2}{2(1-t)} \rt) \Phi\lt( \frac{z-r}{\sqrt{1-t}}   \rt) & \text{ if } z < b_r(t), \\
z & \text{ if } z \geq b_r(t).
\end{array}
\rt.
\end{IEEEeqnarray}
Here $\Phi$ is the cumulative distribution function of a standard normal random variable, $\beta$ is the unique positive solution to 
\begin{IEEEeqnarray*}{rCl}
\sqrt{2\pi} (1-\beta^2) e^{\frac{\beta^2}{2}} \Phi(\beta) = \beta,
\end{IEEEeqnarray*}
and $b_r(t) = r+ \beta \sqrt{1-t}$ is a square-root boundary (approximately, $\beta \approx 0.839924$).
Moreover, as $x \mapsto \sqrt{2\pi} (1-x^2) e^{\frac{x^2}{2}} \Phi(x) - x$ is continuous and monotone 
on $(0,\infty)$, the value of $\beta$, and hence also the function $v_r$, can be calculated to any desired precision.
Furthermore, the region 
\[\mathcal{D}_r := \{(t,z) \in [0,1]\times \R \,:\, z = v_r(t,z) \}=\{(t,z) \in [0,1]\times \R \,:\, z\geq r+\beta \sqrt{1-t}\}\] 
is an optimal stopping region, i.e.~ $\tau_{\mathcal{D}_r}$ is an optimal stopping time. A depiction of this optimal strategy appears in 
Figure~\ref{F:BBS}.

\begin{figure}[h!] \label{F:BBS}
\centering 
\scalebox{0.8}{
\begin{overpic}[scale=0.5, tics=5]{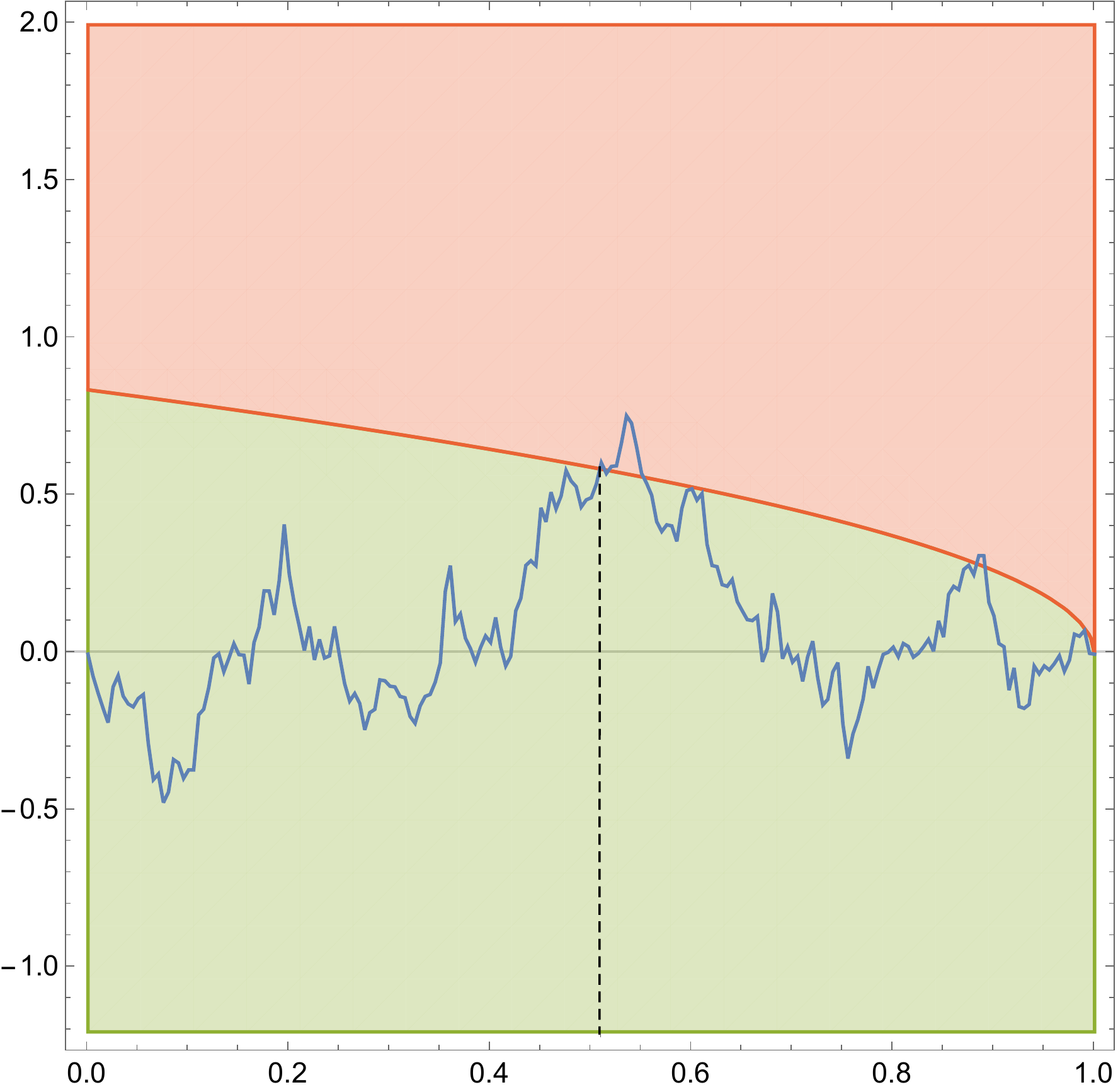}
    \put(52,0.6){\large $\tau_{\mathcal{D}_r}$}
    \put(60,75){\LARGE $\mathcal{D}_r$}
    \put(60, 20){\LARGE $\mathcal{C}_r$}
    \put(7, 47){ \small $t\mapsto Z_t(\omega)$}
   \put(50, -4){$t$}
   \put(-3.5, 50){$z$}
   \put(25,61.6){\small$b_r(t)=\beta\sqrt{1-t}$}
\end{overpic}}
\vspace{10pt}
\caption{A sample trajectory of a standard Brownian bridge $Z$, the stopping region $\mathcal{D}_r$ in pink, the continuation region $\mathcal{C}_r$ in green, as well as the boundary $b_r$ in red; the case $r=0$. }
\label{F:BBS}
\end{figure}
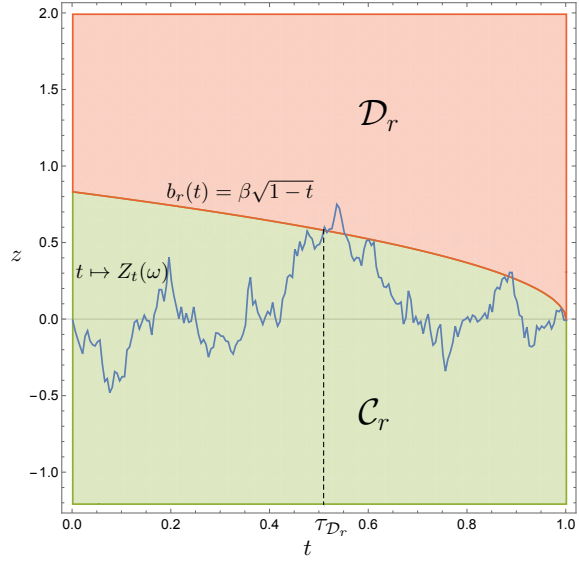

\subsection{Filtering equations}

We now return to the set-up with a general prior distribution as described in the beginning of Section~\ref{S:model}.
Our first result explains how to calculate the posterior distribution given observations of the underlying process. 

\begin{proposition}
\label{T:filtering}
Assume that $Z_0=0$.
Then $\P( X\in \cdot \,|\,\F^Z_t)=\P(X\in \cdot \,|\,Z_t)$, and
$\P(X\in \cdot \,|\,Z_t=z)=\mu_{t,z}(\cdot)$, where
\begin{IEEEeqnarray}{rCl}\label{E:mu}
\mu_{t,z}(\vd u) := \frac{ e^{\frac{uz}{1-t}-\frac{tu^{2}}{2(1-t)}} \mu(\vd u)}{\int_{\R} e^{\frac{uz}{1-t}-\frac{tu^{2}}{2(1-t)}} \mu(\vd u)}.
\end{IEEEeqnarray}
In particular, $\hat X_t := \E[ X\,|\, \F^Z_t]= h(t,Z_t)$, where the function $h$ is given by 
\begin{IEEEeqnarray}{rCl}\label{E:h}
h(t,z) := \frac{\int_{\R}u e^{\frac{uz}{1-t}-\frac{tu^{2}}{2(1-t)}} \mu(\vd u)}{\int_{\R} e^{\frac{uz}{1-t}-\frac{tu^{2}}{2(1-t)}} \mu(\vd u)}.
\end{IEEEeqnarray}
Moreover,
\begin{IEEEeqnarray}{rCl} \label{E:ZP}
\vd Z_t = \frac{ \hat X_t - Z_t}{1-t} \ud t + \ud \hat W_t\,, \;  0 \leq t < 1,
\end{IEEEeqnarray}
where $\hat W$ is a $\mathcal{F}^Z$-Brownian motion.
\end{proposition}

\begin{proof}
Define $Y_{s} := (1+s)Z_{t(s)}$, where $t(s) = s/(1+s)$ and $s\geq 0$. Then
\begin{IEEEeqnarray*}{rCl}
\vd Y_{s} &=& Z_{t(s)} \ud s + (1+s) \ud Z_{t(s)} \\
&=& Z_{t(s)} \ud s + (1+s) \lt( \frac{X-Z_{t(s)}}{1-t(s)} \ud t(s) + \vd W_{t(s)} \rt)\\
&=& X \ud s + \ud B_{s},
\end{IEEEeqnarray*}
where $B_s:=\int_0^{t(s)}\frac{1}{1-u} \ud W_u$ is a Brownian motion. 
From filtering theory (see \cite[Proposition~3.1]{BC09}),
$\P( X\in \cdot \,|\,\F^Y_s)=\P(X\in \cdot \,|\,Y_s)$, and
$\P(X\in \cdot \,|\,Y_s=y)=\nu_{s,y}(\cdot)$, where
\begin{IEEEeqnarray}{rCl} \label{E:FE}
\nu_{s,y}(\vd u) := \frac{ e^{uy-\frac{su^{2}}{2}} \mu(\vd u)}{\int_{\R} e^{uy-\frac{su^{2}}{2}} \mu(\vd u)}.
\end{IEEEeqnarray}
In particular, 
\begin{IEEEeqnarray}{rCl} \label{E:deff}
\E\lt[ X \,|\, \F^{Y}_{s} \rt] =\E\lt[X \,|\, {Y}_{s} \rt]= f(s,Y_{s}),
\end{IEEEeqnarray}
where
\begin{IEEEeqnarray}{rCl} \label{E:FF}
f(s,y) = \frac{\int_{\R} u e^{uy-u^{2}s/2} \mu(\vd u)}{\int_{\R} e^{uy-u^{2}s/2} \mu(\vd u)}.
\end{IEEEeqnarray}
Moreover, 
\begin{IEEEeqnarray}{rCl}\label{E:Y}
\vd Y_s 
&=& f(s, Y_s) \ud t + \ud \hat B_s, \IEEEyesnumber \label{E:Yeqn}
\end{IEEEeqnarray}
where $\hat B_s := \int_0^s (X-\E [X \,|\, \F^Y_u]) \ud u + B_s$ is an $\F^Y$-Brownian motion (see \cite[Proposition 2.30 on p.~33]{BC09}), known as the innovation process. Thus, recalling that $Z_t = (1-t) Y_{\frac{t}{1-t}}$, $t \in [0, 1)$, we get that
\begin{IEEEeqnarray*}{rCl} 
\vd Z_t = \frac{ \hat X_t - Z_t}{1-t} \ud t + \ud \hat W_t\,, \;  0 \leq t < 1,
\end{IEEEeqnarray*}
where $\hat W_t=\int_0^{t/(1-t)}\frac{1}{1+s}\ud \hat B_s$ is a $\mathcal{F}^Z$-Brownian motion on $[0,1)$.
Here $\hat X_t := \E[ X\,|\, \F^Z_t]= h(t,Z_t)$, where the function $h$ satisfies $h(t, z) = f(\frac{t}{1-t}, \frac{z}{1-t})$. 
Also, the identity \eqref{E:FE} tells us that $\P( X\in \cdot \,|\,\F^Z_t)=\P(X\in \cdot \,|\,Z_t)$ and that
\begin{IEEEeqnarray*}{rCl}
\P(X\in \cdot \,|\,Z_t=z)=\mu_{t,z}(\cdot),
\end{IEEEeqnarray*}
where
\begin{IEEEeqnarray*}{rCl}
\mu_{t,z}(\vd u) := \frac{ e^{\frac{uz}{1-t}-\frac{tu^{2}}{2(1-t)}} \mu(\vd u)}{\int_{\R} e^{\frac{uz}{1-t}-\frac{tu^{2}}{2(1-t)}} \mu(\vd u)}.
\end{IEEEeqnarray*}
\end{proof}

\begin{remark}
Note that, by the above, an equivalent reformulation of \eqref{E:OSR} is
\begin{IEEEeqnarray}{rCl} \label{E:OSY}
V=\sup_{\tau \in \mathcal{T}^{Y}} \E \lt[\frac{1}{1+\tau} Y_{\tau} \rt],
\end{IEEEeqnarray}
where $\mathcal{T}^{Y}$ denotes the set of stopping times with respect to the filtration generated by $Y$ 
given by \eqref{E:Y}, and where $Y_\tau/(1+\tau):=\lim_{t\to\infty}Y_t/(1+t)$ on $\{\tau=\infty\}$. In fact, this formulation was used by Shepp in his study \cite{lS} of the Brownian bridge with a known
pinning point.
\end{remark}

From here on, we work under the following assumption.

\begin{assumption}
\label{ass}
The variance of the posterior distribution $\mu_{t,z}$ in \eqref{E:mu} is bounded on $[0,t_0]\times\R$ for any fixed $t_0\in[0,1)$.
\end{assumption}

Since $\hat X=h(t,Z_t)$, the equation \eqref{E:ZP} gives a description of the bridge $Z$ as the solution of a stochastic differential equation. A consequence of Assumption~\ref{ass} is that the drift of $Z$ is Lipschitz continuous in the spatial variable on $[0,t_0]\times\R$ for any $t_0\in[0,1)$.

\begin{proposition}
\label{T:Lip}
The function $h(t,z)$ (and therefore also the drift $(h(t,z)-z)/(1-t)$ of $Z$) is Lipschitz continuous in $z$, uniformly in $t\in[0,t_0]$ for any $t_0\in[0,1)$.
\end{proposition}

\begin{proof}
Straightforward differentiation shows that 
\begin{IEEEeqnarray*}{rCl} 
\frac{\partial h(t,z)}{\partial z} &=& \frac{\int_{\R}u^2 e^{\frac{uz}{1-t}-\frac{tu^{2}}{2(1-t)}} \mu(\vd u)}{(1-t)\int_{\R} e^{\frac{uz}{1-t}-\frac{tu^{2}}{2(1-t)}} \mu(\vd u)}-\frac{1}{1-t}\left(\frac{\int_{\R}u e^{\frac{uz}{1-t}-\frac{tu^{2}}{2(1-t)}} \mu(\vd u)}{\int_{\R} e^{\frac{uz}{1-t}-\frac{tu^{2}}{2(1-t)}} \mu(\vd u)}\right)^2\\
&=& \frac{1}{1-t}\left( \int_{\R}u^2  \mu_{t,z}(\vd u)-\left( \int_{\R}u  \mu_{t,z}(\vd u)\right)^2\right),
\end{IEEEeqnarray*}
which is bounded by Assumption~\ref{ass}.
\end{proof}

\subsection{Markovian embedding}

Thanks to Proposition~\ref{T:filtering}, we can view
the original problem \eqref{E:OSR} as an optimal stopping problem for a process $Z$ having an SDE representation 
\begin{IEEEeqnarray*}{rCl} 
\left\{\begin{array}{ll}
\vd Z_t = \frac{ h(t,Z_t) - Z_t}{1-t} \ud t + \ud \hat W_t\,, \;  0 \leq t < 1,\\
Z_0=0.\end{array}\right.
\end{IEEEeqnarray*}
To study this stopping problem we embed it into a Markovian framework, i.e. we
define a Markovian value function 
\begin{IEEEeqnarray}{rCl} \label{E:OSMZ}
v(t,z) = \sup_{\tau \in \mathcal{T}^Z_{1-t}} \E\lt[ Z^{t,z}_{t+\tau} \rt], \quad (t,z) \in [0,1) \times \R,
\end{IEEEeqnarray}
where $\mathcal{T}^Z_{1-t}$ denotes the set of stopping times with respect to the process $Z=Z^{t,z}$ defined by 
\begin{IEEEeqnarray}{rCl} \label{E:ZMark}
\left\{
\begin{array}{ll}
\vd Z_{t+s} = \frac{ h(t+s,Z_{t+s}) - Z_{t+s}}{1-(t+s)} \ud s + \ud \hat W_{t+s}\,, &  0 \leq s < 1-t, \\
Z_t=z, & z\in \R.
\end{array}
\right.
\end{IEEEeqnarray}
Clearly, $v(t,z) \geq z$ for all $(t,z) \in [0,1)\times \R$.
Let us define the continuation region
\begin{IEEEeqnarray*}{rCl}
\mathcal{C} &:=& \{ (t,z) \in [0, 1)\times \R \,:\, v(t,z) > z \}
\end{IEEEeqnarray*}
and the stopping region
\begin{IEEEeqnarray*}{rCl}
\mathcal{D} &:=& \lt([0,1]\times \R \rt)\setminus \mathcal{C}.
\end{IEEEeqnarray*}
In addition, let $\tau_{\mathcal{D}} :=\inf \{s\geq0 \,:\, (t+s,Z^{t,z}_{t+s}) \in \mathcal{D}\}$ be the first entry time to the stopping region. 

%

\begin{theorem}
\label{T:OS} The following properties hold.
\begin{enumerate}[(i)]
\item
The stopping time $\tau_{\mathcal{D}}$ is optimal for the optimal stopping problem \eqref{E:OSMZ}.
\item
The value function $v$ is continuous on $[0,1)\times\R$.
\item
The value function $v$ satisfies
\begin{IEEEeqnarray}{rCl}\label{E:fbp}
\lt\{
\begin{array}{ll}
\pt_1 v + \mathcal L v = 0, & (t,z) \in \mathcal{C}, \\
v(t,z) = z, & (t,z) \in \mathcal{D},
\end{array}
\rt.
\end{IEEEeqnarray}
where $\mathcal L=\frac{h(t,z)-z}{1-t} \pt_2  + \frac{1}{2} \pt^2_2 $.
\end{enumerate}
\end{theorem}

\begin{proof}
\begin{enumerate}[(i)]
\item
Since the pinning point $X$ is integrable, we have that 
\begin{IEEEeqnarray*}{rCl} 
\E\left[\sup_{0\leq s\leq 1-t} \vert Z_{t+s}\vert\right]\leq \E\left[\sup_{0\leq s\leq 1-t} \vert Z'_{t+s}\vert + \vert X\vert\right]<\infty,
\end{IEEEeqnarray*}
where $Z'$ is a Brownian bridge pinning at 0. Consequently, standard results from optimal stopping theory (see \cite[Appendix D]{KS2} or \cite{PS06}) imply that $\tau_{\mathcal{D}}$ is optimal.
\item
For the continuity of $v$, first define
\begin{IEEEeqnarray*}{rCl} 
v_\ep(t,z) = \sup_{\tau \in \mathcal{T}^Z_{1-t-\ep}} \E\lt[ Z^{t,z}_{t+\tau} \rt], \quad (t,z) \in [0,1) \times \R.
\end{IEEEeqnarray*}
Then $v_\ep$ is the value function in the case when stopping is restricted to the time interval $[t,1-\ep]$.
On this interval, the drift of $Z$ is Lipschitz in the spatial variable by Proposition~\ref{T:Lip},
so standard methods can be applied to show that 
$v_\ep$ is continuous (see for example \cite[Chapter 3]{K}). 

Next, letting $v^\ep$ denote the value function in a set-up where all information (i.e. the pinning point) is revealed 
at time $1-\ep$; we clearly have $v\leq v^\ep$. Moreover, $v^\ep$ is the value function of an optimal stopping problem
for $Z$ with horizon $T:=1-\ep$ and payoff function 
\begin{IEEEeqnarray*}{rCl} 
g^\ep(t,z)=\left\{\begin{array}{ll}
z, & t\in[0,T),\\
\int_\R v_r(T,z)\mu_{1-\ep,z}(\ud r), & t=T,\end{array}\right.
\end{IEEEeqnarray*}
where $v_r$ is the value function corresponding to stopping a Brownian motion pinning at $r$ at time 1, see
\eqref{E:known}. 
It is straightforward to check that the Lipschitz property of $r\mapsto v_r(t,z)$ with $(t, z)$ fixed together with Assumption~\ref{ass} imply that $g^\ep$ is Lipschitz continuous, and hence standard methods can again be applied to show that $v^\ep$ is continuous.
Moreover, note that
\begin{IEEEeqnarray*}{rCl} 
0 &\leq& v^\ep(t,z)-v_\ep(t,z) \leq \E\left[\sup_{1-\ep-t\leq s\leq 1-t}Z^{t,z}_{t+s}-Z^{t,z}_{1-\ep}\right]\to 0
\end{IEEEeqnarray*}
as $\ep\to 0$ by dominated convergence.
Since $v$ is sandwiched between the continuous functions $v_\ep$ and $v^\ep$, it follows that $v$ is continuous. 
\item
The fact that $v$ satisfies \eqref{E:fbp} is a standard consequence of the continuity of $v$ and the Markov property, 
see e.g. \cite[Theorem 2.7.7]{KS2}.
\end{enumerate}
\end{proof}

\begin{remark}
It is important to note that $Z^{t,z}$ can be interpreted as a Brownian bridge started at time $t$ from $z$ and with the pinning point having the prior $\mu_{t,z}$. 
Theorem~\ref{T:OS} yields that by determining $\mathcal{C}$ and $\mathcal{D}$ we solve the optimal stopping problem for every such Brownian bridge $Z^{t,z}$. In this way, the continuation region and the stopping region depend on both the starting point and the prior distribution.
In particular, the regions $\mathcal C$ and $\mathcal D$ would change if the starting point $z$ was shifted without altering the 
prior distribution accordingly.
\end{remark}


We end this section by introducing some terminology useful for interpreting stopping decisions.
If $(t,z) \in \mathcal{D}$, we call $(t,z)$ a \emph{stopping point}. We will say that a stopping point is 
\begin{itemize}
\item
a \emph{stop-loss} if, for all $\epsilon > 0$, the intersection $\{(t,z')\,:\, z'\in (z, z+\epsilon) \}\cap\mathcal{C}\neq \emptyset$;
\item
\emph{too-good-to-persist}, if, for all $\epsilon > 0$, the intersection $\{(t,z')\,:\, z'\in (z-\epsilon, z) \}\cap\mathcal{C}\neq \emptyset$.
\end{itemize}

\section{Structural properties}

In this section, we provide some structural properties of the continuation region and the stopping region. We first give bounds for the 
value function, which translate into lower estimates of the stopping region and the continuation region.

\subsection{Bounds for the value function}

%

\begin{proposition}[Bounds for the value function] \label{T:VEC0}
For $r\in\R$, let $v_r$ and $v_{-r}$ be the value functions (given in \eqref{E:known} above) for known pinning points $r$ and $-r$, respectively. 
\begin{enumerate}[(i)]
\item \label{T:vrC}
If $\supp \mu\subseteq (-\infty,r]$, then $v \leq v_r.$ 
Consequently, $\mathcal{D}_r:=\{ (t,z) \in [0,1]\times \R \,:\, z \geq  r + \beta \sqrt{1-t}\} \subseteq \mathcal{D}$.
\item \label{T:v-rC}
If $\supp \mu\subseteq [-r,\infty)$, then $v \geq v_{-r}.$ Consequently,
$\mathcal{C}_{-r}:=\{ (t,z) \in [0,1)\times \R \,:\, z <  -r + \beta \sqrt{1-t}\} \subseteq \mathcal{C}$.
\end{enumerate}
\end{proposition}

\begin{proof}
\begin{enumerate}
\item[(i)]
Let $t\in [0,1)$, $z\in \R$, and denote by 
\begin{IEEEeqnarray*}{rCl}
\mathcal L_r=\frac{r-z}{1-t} \pt_2  + \frac{1}{2} \pt^2_2 
\end{IEEEeqnarray*}
the differential operator associated with a bridge pinning at $r$. Applying It\^{o}'s formula to $v_r(t+s, Z^{t,z}_{t+s})$ and taking expectations at a stopping time $\tau \in \mathcal{T}^Z_{1-t}$, we have
\begin{IEEEeqnarray*}{rCl}
v_r(t,z) &=& \E\lt[ v_r(t+ \tau, Z^{t,z}_{t+\tau}) \rt] - \E\lt[ \int_0^\tau \pt_2 v_r(t+u, Z^{t,z}_{t+u})\ud \hat W_{t+u} \rt]\\
&&- \E \bigg[ \int_0^\tau \big(\pt_1 v_r+\mathcal L v_r)(t+u, Z^{t,z}_{t+u})  
 \Ind_{\{  Z^{t,z}_{t+u}   \neq b_r(t+u)\}} \ud u \bigg] \\
&\geq& \E \lt[ v_r(t+\tau, Z^{t,z}_{t+\tau}) \rt] \\
&&- \E \bigg[ \int_0^\tau \big(\pt_1 v_r+\mathcal L_r v_r)(t+u, Z^{t,z}_{t+u})  
 \Ind_{\{  Z^{t,z}_{t+u}   \neq b_r(t+u)\}} \ud u \bigg] \\
&\geq& \E\lt[ v_r(t+\tau, Z^{t,z}_{t+\tau}) \rt]\geq  \E\lt[ Z^{t,z}_{t+\tau} \rt]. \IEEEyesnumber \label{E:vrZ}
\end{IEEEeqnarray*}
In the above, the first inequality follows from the fact that $h \leq r$ and an application of the optional sampling theorem, the second inequality follows from the fact that the integrand inside the second expectation is non-negative, and the last one 
holds because $v_r(t,z)\geq z$ for all $(t,z) \in [0,1]\times\R$.  
Hence, since $\tau$ was arbitrary, from \eqref{E:vrZ} we obtain
\begin{IEEEeqnarray*}{rCl}
v_r(t,z) \geq \sup_{\tau \in \mathcal{T}^Z_{1-t}} \E[Z^{t,z}_{t+\tau}] = v(t,z).
\end{IEEEeqnarray*}
\item[(ii)]

Let $Z$ be a Brownian bridge started at time $t\in [0,1)$ from $Z_t=z$ and with the pinning point $Z_1 \sim \mu_{t,z}$. Also, let $Z'$ be  a Brownian bridge started at time $t\in [0,1)$ from $Z'_t=z$ and pinning at $Z'_1=-r$. Moreover, we can choose a probability space so that
\begin{IEEEeqnarray*}{rCl}
Z_{t+s} \geq Z'_{t+s} \quad \text{for all $s\in [0,1-t]$.} 
\end{IEEEeqnarray*}
Consequently, taking $\tau_{D_{-r}} = \inf \{s\geq0\,:\, Z_{t+s}\in \mathcal{D}_{-r} \}$ and $\tau'_{\mathcal{D}_{-r}}= \inf \{s \geq\,:\, Z'_{t+s} \in \mathcal{D}_{-r} \}$, we have $\tau_{D_{-r}} \leq\tau'_{\mathcal{D}_{-r}}$. Therefore, since $t\mapsto -r+ \beta \sqrt{1-t}$ is decreasing,  
\begin{IEEEeqnarray*}{rCl}
v(t,z) &\geq& \E[Z_{t+\tau_{\mathcal{D}_{-r}}}] \geq \E[Z'_{t+\tau'_{\mathcal{D}_{-r}}}] = v_{-r}(t,z).
\end{IEEEeqnarray*}
\end{enumerate}
\end{proof}

\subsection{A condition for a single-upper boundary strategy}

Recognising the complicated geometry of the stopping region for a general prior, we now turn to the problem of
identifying a condition on the prior that guarantees a one-sided stopping region.

\begin{theorem}[Condition for a single stopping boundary] \label{T:DY}
\item
\begin{enumerate}[(i)]
\item \label{Tp:DDec}
If $h(t,z)-z$ is decreasing in $z$ (equivalently, $f(t,y) - y/(1+t)$ is decreasing in $y$), then there exists a boundary $b:[0,1) \to [-\infty,\infty]$ such that $\mathcal{C} = \{ (t,z) \in [0,1) \times \R \,:\, z < b(t) \}$.
\item \label{Tp:DInc}
If $h(t,z)-z$ is increasing in $z$ (equivalently, $f(t,y) - y/(1+t)$ is increasing in $y$), then there exists a boundary $b:[0,1) \to [-\infty,\infty]$ such that $\mathcal{C} = \{ (t,z) \in [0,1) \times \R \,:\, z > b(t) \}$.
\end{enumerate}
\end{theorem}

\begin{proof}
\begin{enumerate}
\item[(i)]
Since
\begin{IEEEeqnarray*}{rCl}
\vd Z_s &=& \frac{h(s,Z_s)-Z_s}{1-s} \ud s + \ud \hat W_s,
\end{IEEEeqnarray*}
we have for a given stopping time $\tau$ that
\begin{IEEEeqnarray}{rCl}\label{E:single}
\E \lt[  Z^{t,z}_{t+\tau} \rt] &=& z +\E\lt[\int_0^{\tau}\frac{h(t+s,Z_{t+s}^{t,z})-Z^{t,z}_{t+s}}{1-t-s}\ud s \right]
\end{IEEEeqnarray}
by optional sampling.
Let $z_2 > z_1$. Then the trajectories of $Z^{t,z_2}$ and $Z^{t,z_1}$ do not cross before the end of 
time, so, by continuity, they never cross. Therefore \eqref{E:single} and 
the monotone decay of $h(t,z) - z$ in $z$ imply that 
\begin{IEEEeqnarray*}{rCl}
\E \lt[ Z^{t,z_1}_{t+\tau}\rt] -z_1\geq\E \lt[ Z^{t,z_2}_{t+\tau}\rt] -z_2.
\end{IEEEeqnarray*}
Taking the supremum over stopping times then yields
\begin{IEEEeqnarray*}{rCl}
v(t,z_1) -z_1  \geq v(t,z_2) -z_2.
\end{IEEEeqnarray*}
Consequently, $(t,z_2)\in \mathcal C$ implies $v(t,z_1) - z_1  \geq v(t,z_2)  -z_2>0$, so
then also $(t,z_1)\in\mathcal C$, which finishes the proof of the claim.
\item [(ii)]
The proof of the second claim is analogous to the first one, and therefore omitted.
\end{enumerate}
\end{proof}

One consequence of the structural property in Theorem~\ref{T:DY} is that the stopping region can be characterized by 
an integral equation for its boundary. To avoid technical considerations in connection with the general case, however, we choose to 
discuss this only in the special case of the normal prior, see Section~\ref{S:normal} below.

\section{Rich structure already in a simple case: the two-point distribution}

In this section, we study the two-point prior distribution case, i.e. the prior $\mu = p \delta_r + (1-p) \delta_l$, where $r>l$, $p \in (0,1)$. 
In what follows, we will work within the Markovian framework, allowing the process $Z$ to start at any value. Consequently, without loss of generality, we can assume that $\mu = p \delta_r + (1-p) \delta_{-r}$, where $r>0$, so that the support of $\mu$ is symmetric around 0.

By Proposition~\ref{T:filtering} we have
\begin{IEEEeqnarray*}{rCl}
\P(X=r \,|\, \F^Z_t)=\P(X=r \,|\, Z_t) = \pi (t,Z_t),
\end{IEEEeqnarray*}
where
\begin{IEEEeqnarray*}{rCl}
\pi(t,z) := \frac{\frac{p}{1-p}e^{\frac{2rz}{1-t}}}{\frac{p}{1-p}e^{\frac{2rz}{1-t}}+1}.
\end{IEEEeqnarray*}

\begin{proposition}[Lower bound for the value function] \label{T:VEC}
We have
\begin{IEEEeqnarray}{rCl}\label{E:lowerbound}
v(t,z)\geq \pi(t,z)  v_r(t,z) - (1-\pi(t,z))r .
\end{IEEEeqnarray}
Consequently, $Q_r:=\{ (t,z) \in [0,1)\times \R \,:\, z< \pi(t,z) v_r(t,z) - (1-\pi(t,z))r \} \subseteq \mathcal{C}$.
\end{proposition}

\begin{proof}
The value $v(t,z)$ corresponds to optimally stopping a Brownian bridge started at time $t$ from $Z_t=z$ when the pinning point has the distribution $\pi(t,z)\delta_r + (1-\pi(t,z))\delta_{-r}$. Since the stopping time $\tau_{D_r}= \inf \{s \in [1-t]\,:\, (t+s,Z^{t,z}_{t+s})\in \mathcal{D}_r\}$ is  suboptimal, 
\begin{IEEEeqnarray*}{rCl}
v(t,z)&\geq& \E[Z^{t,z}_{t+\tau_{\mathcal{D}_r}}] \\
&=&  \E[Z^{t,z}_{t+\tau_{\mathcal{D}_r}}\vert X=r]\P(X=r)+\E[Z^{t,z}_{t+\tau_{\mathcal{D}_r}}\vert X=-r]\P(X=-r)\\
&\geq& \pi(t,z)  v_r(t,z) - (1-\pi(t,z))r .
\end{IEEEeqnarray*}
\end{proof}

We point out that the explicitness of the function $v_r$ makes it easy to plot the region $Q_r$, compare Figures 2 and 3 in which 
the region $Q_r$ and the regions $\mathcal D_r$ and $\mathcal C_{-r}$ are drawn for different sets of parameters.

\begin{remark}
An alternative way to find continuation points would be to use the 
(suboptimal) strategy $\tau=1-t$, i.e. to simply continue until the end of time.
This strategy, however, would yield the lower bound $\pi(t,z)r-(1-\pi(t,z))r$, which is also implied by
\eqref{E:lowerbound} using $v_r(t,z)\geq r$.
Another standard way to find a lower bound on the continuation region is to apply the infinitesimal generator of the underlying 
process to the payoff function, compare \cite[Chapter 10]{bO}. In the current case, however, that method would also yield 
the weaker lower bound $\pi(t,z)r-(1-\pi(t,z))r$ for the value function.
\end{remark}

\begin{figure}[!h]
\begin{subfigure}[b]{0.45\textwidth}
\begin{overpic}[scale=0.5, tics=5]{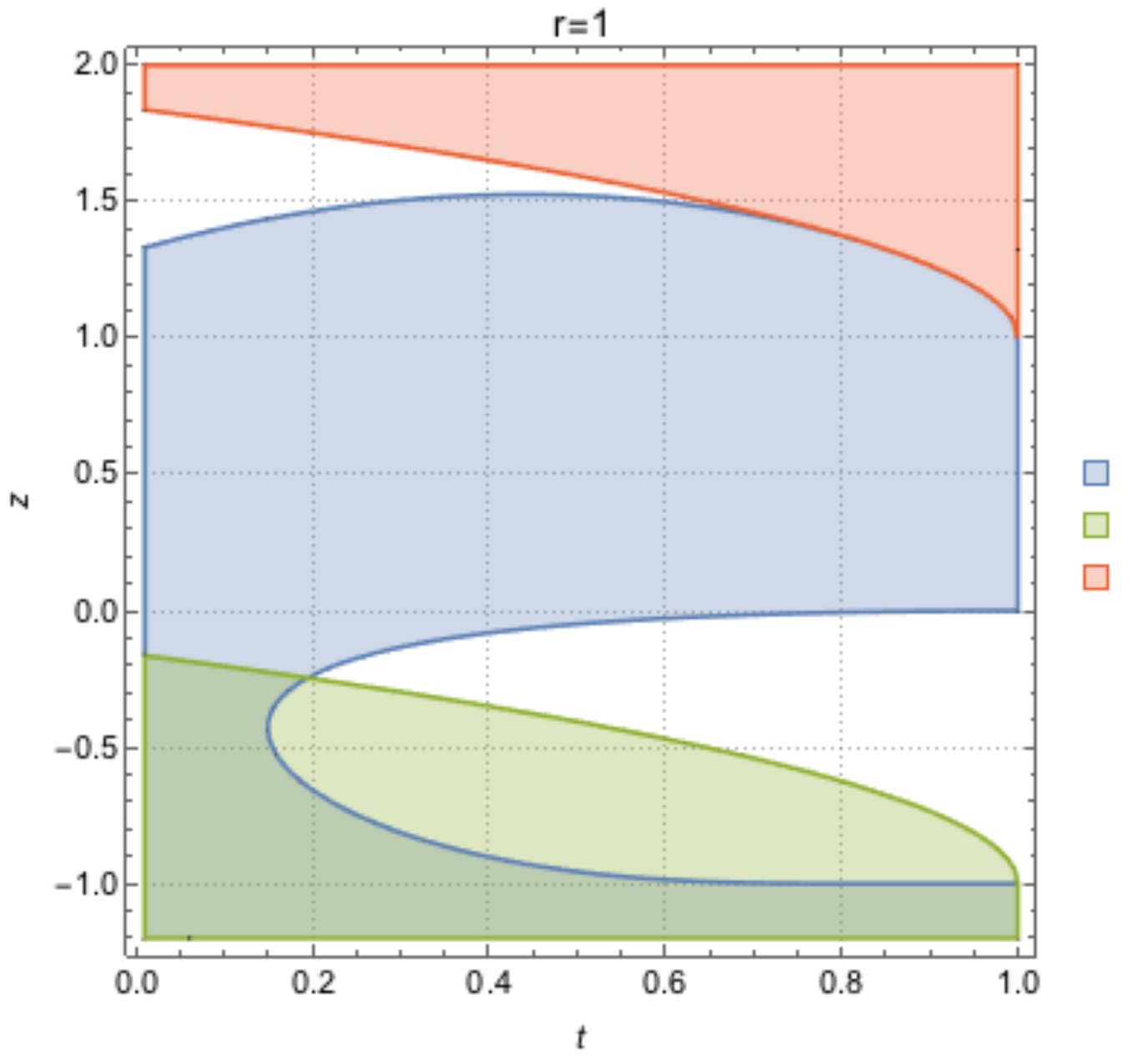}
\put(98, 50.6){\tiny $Q_r$}
\put(98, 46.0){\tiny $\mathcal{C}_{-r}$}
\put(98, 41.4){\tiny $\mathcal{D}_r$}
\end{overpic}
\caption{$r=1$}
\end{subfigure} \qquad
\begin{subfigure}[b]{0.45\textwidth}
\begin{overpic}[scale=0.5, tics=5]{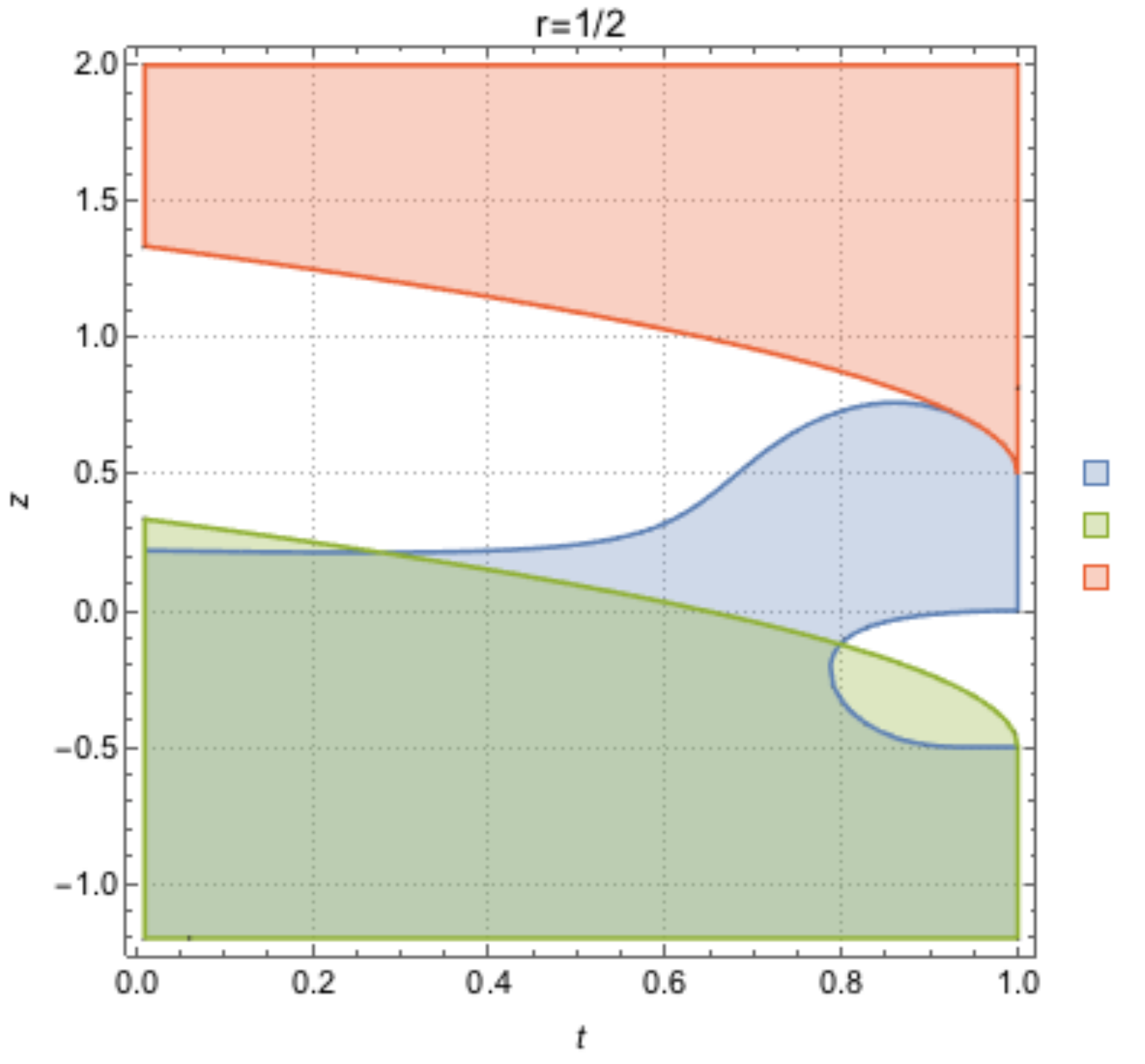}
\put(98, 50.6){\tiny $Q_r$}
\put(98, 46.0){\tiny $\mathcal{C}_{-r}$}
\put(98, 41.4){\tiny $\mathcal{D}_r$}
\end{overpic}
\caption{$r=1/2$}
\end{subfigure}
\caption{The regions $Q_r$, $\mathcal{C}_{-r}$, and $ \mathcal{D}_r$ in the symmetrically weighted case, i.e.~$p=1/2.$ Note that blue and green regions are subsets of $\mathcal C$, and the red region is contained in $\mathcal D$. The white areas consist of points of unknown 
type; by Proposition~\ref{T:RKP}, however, the lower white region is known to contain stopping points.}
\end{figure}

\begin{figure}[h]
\begin{subfigure}[b]{0.45\textwidth}
\begin{overpic}[scale=0.5, tics=5,trim=0 0 0 0, clip=true]{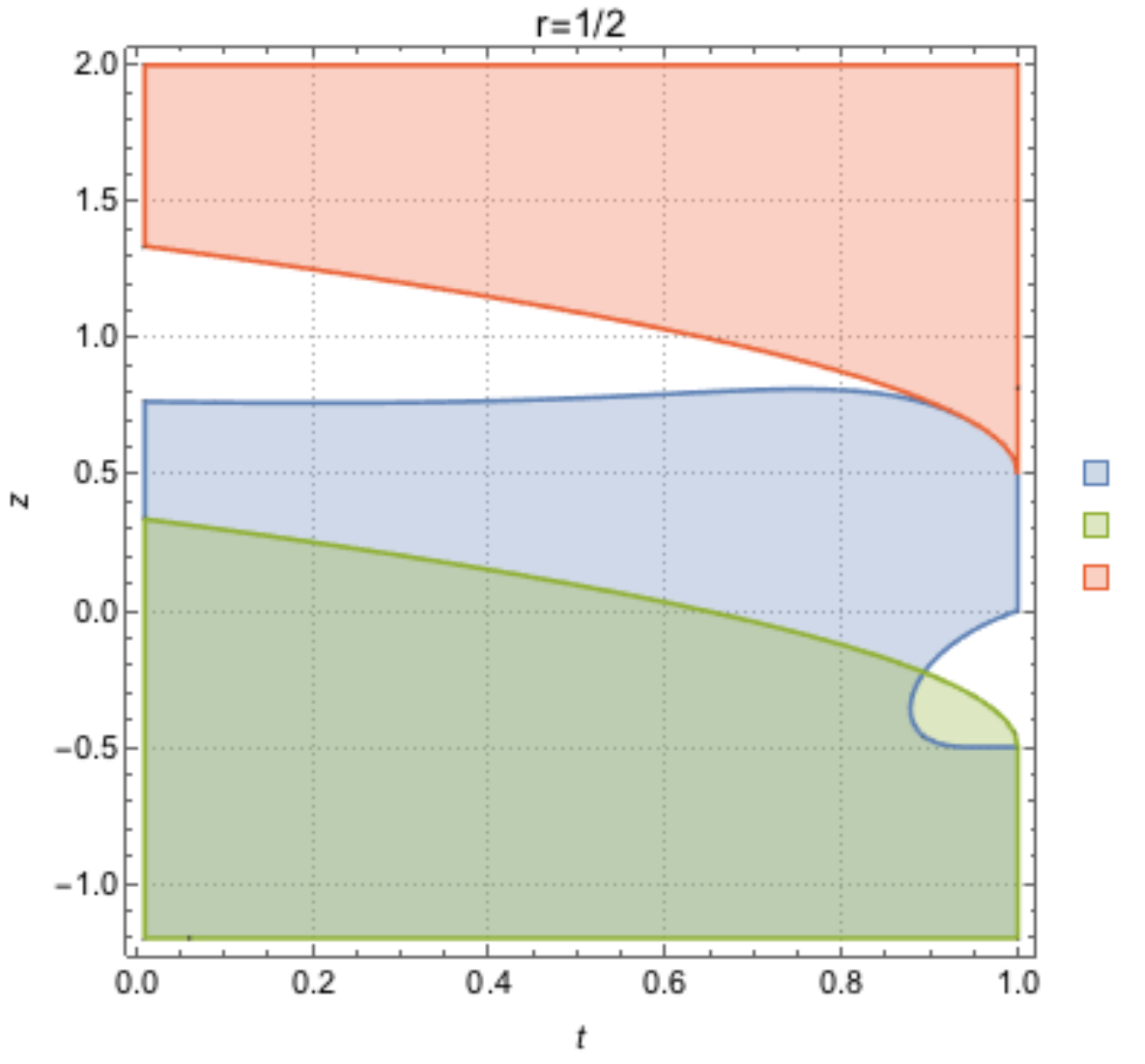}
\put(98, 50.6){\tiny $Q_r$}
\put(98, 46.0){\tiny $\mathcal{C}_r$}
\put(98, 41.4){\tiny $\mathcal{D}_r$}
\end{overpic}
\caption{$p=3/4$}
\end{subfigure} \qquad
\begin{subfigure}[b]{0.45\textwidth}
\begin{overpic}[scale=0.5, tics=5]{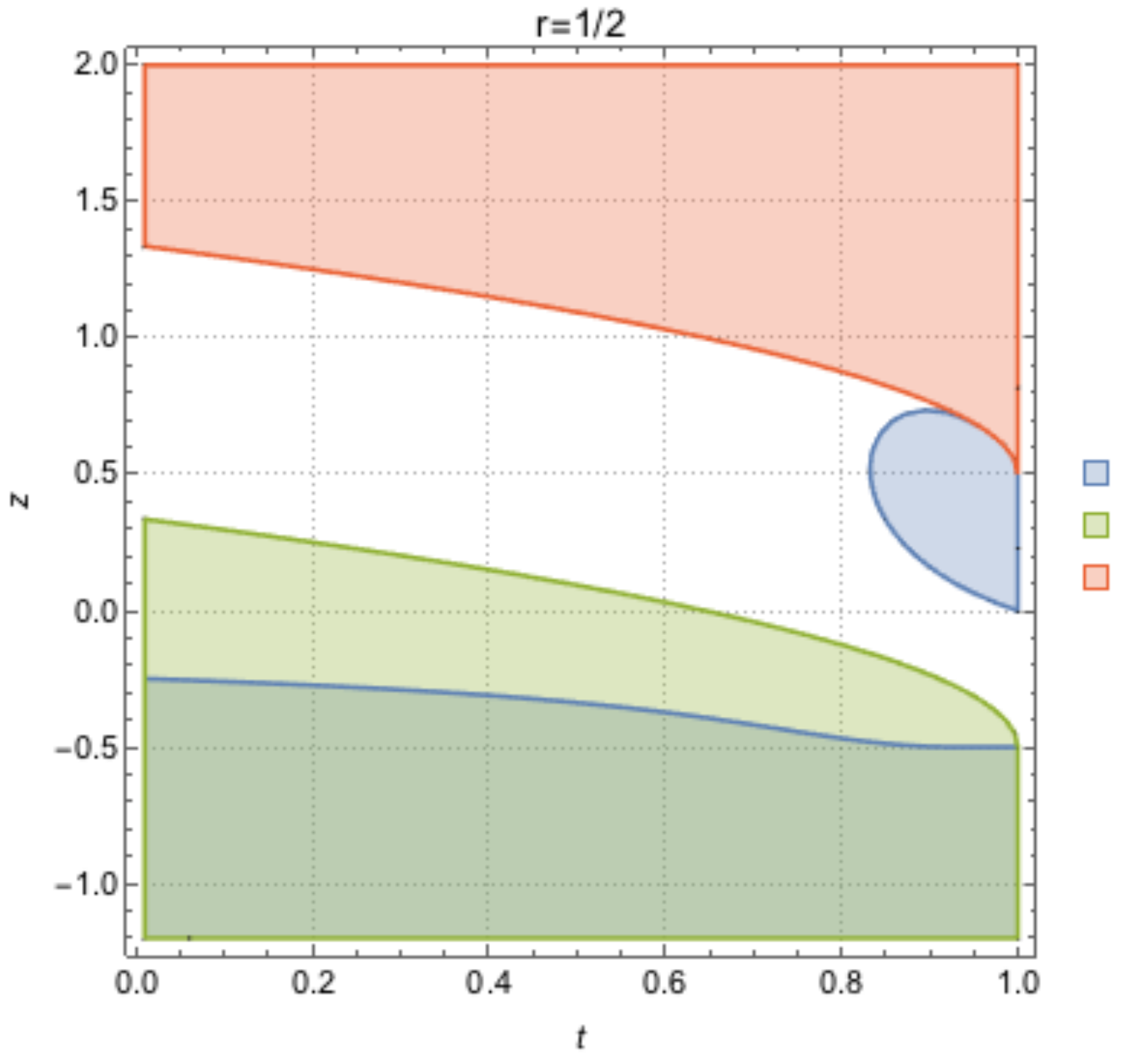}
\put(98, 50.6){\tiny $Q_r$}
\put(98, 46.0){\tiny $\mathcal{C}_r$}
\put(98, 41.4){\tiny $\mathcal{D}_r$}
\end{overpic}
\caption{$p=1/4$}
\end{subfigure}
\caption{The regions $Q_r$, $\mathcal{C}_{-r}$, and $ \mathcal{D}_r$ in non-symmetrically weighted cases with $r=1/2$.}
\end{figure}
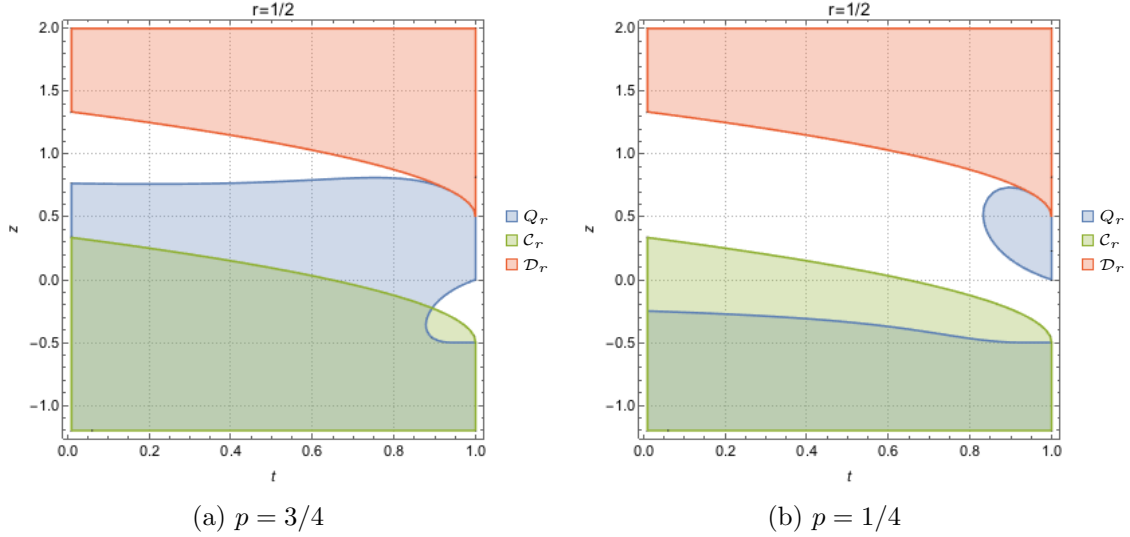

Next, we address the limiting behaviour of the value function close to the terminal pinning time.

\begin{proposition}[Discontinuity of $v$ in the limit] 
For fixed $z\in\R$, the value function satisfies
\begin{IEEEeqnarray}{rCl}\label{E:asymp}
\lim_{t \nearrow 1} v(t, z) =  -r \Ind_{\{ z < -r\}} + z \Ind_{\{-r\leq z < 0 \}} +pr\Ind_{\{z =0\}}+ r \Ind_{\{0<z\leq r \}} + z \Ind_{\{ z>r \}}.
\end{IEEEeqnarray}
\end{proposition}

\begin{remark}
Note that the right-hand side of \eqref{E:asymp} is not continuous at $z=0$. Thus $v$ cannot in general be extended to a continuous function
on $[0,1]\times\R$.
\end{remark}

\begin{proof}
First note that
\begin{IEEEeqnarray}{rCl}\label{E:pi}
\pi(t,z) = \frac{\frac{p}{1-p}e^{\frac{2rz}{1-t}}}{\frac{p}{1-p}e^{\frac{2rz}{1-t}}+1}
\tend \left\{ 
\begin{array}{ll}
1, & z>0, \\
p, & z=0,\\
0, & z<0,
\end{array}\rt.
\end{IEEEeqnarray}
as $t\nearrow 1$. For $z>0$, letting $t\nearrow1$ in
\begin{IEEEeqnarray*}{rCl}
\pi v_r -r (1-\pi) \leq v \leq \pi v_r + (1-\pi)v_{-r}
\end{IEEEeqnarray*}
gives the result (here the lower bound follows from Proposition~\ref{T:VEC} and the upper bound follows from comparison with 
a case with full information about the pinning point).
Similarly, by Proposition~\ref{T:VEC0},
\begin{IEEEeqnarray*}{rCl}
v_{-r}
 \leq v \leq \pi v_r + (1-\pi)v_{-r},
\end{IEEEeqnarray*}
so letting $t\nearrow 1$ yields the desired limit for $z<0$.

Now, consider the remaining case $z=0$. Let $\hat{Z}$, $\check{Z}$ be Brownian bridges pinning at $r$ and $-r$, respectively, and introduce 
$\rho(t) := \P (\inf_{0\leq s \leq 1-t} \hat{Z}^{t,0}_{t+s}  \leq - \epsilon)$. 
Then $\rho(t) \to 0$ as $t\nearrow 1$, so a comparison with the suboptimal strategy 
$\tau_{-\ep}:=\inf\{s\in[0,1-t):Z^{t,0}_{t+s}\leq -\ep\}\wedge(1-t)$ yields
\begin{IEEEeqnarray*}{rCl}
v(t,0) &\geq& \pi(t,0)(1-\rho(t))r - \pi(t,0) \rho(t) \epsilon - (1-\pi(t,0))\epsilon \\
&\to& 
pr - (1-p)\epsilon 
\end{IEEEeqnarray*}
as $t \nearrow 1$. Thus, as $\epsilon$ was arbitrary, 
$\liminf_{t\nearrow1} v(t,0) \geq pr$.

On the other hand, letting $t\nearrow1$ in
\begin{IEEEeqnarray*}{rCl}
v(t,0) \leq \pi v_r(t,0) + (1-\pi)v_{-r}(t,0)
\end{IEEEeqnarray*}
yields $\limsup_{t\nearrow 1} v(t,0) \leq pr$. Thus we can conclude that $\lim_{t\nearrow 1} v(t,0) = pr$.
\end{proof}

\begin{proposition}[Disconnected stopping region] \label{T:RKP} The following properties hold.
\begin{enumerate}
\item [(i)]
Close to the terminal time, there is a region of continuation points as follows: for every $\ep>0$ there
exists $\delta>0$ such that $(1-\delta,1)\times (\ep,r-\ep)\subseteq \mathcal C$.
\item [(ii)]
There is a region close to the terminal time in which stopping points are dense in the following sense: for all $z \in (-r, 0)$ and for all $\epsilon >0$, the intersection $[1-\epsilon, 1) \times (z-\epsilon, z+\epsilon) \cap \mathcal{D} \neq \emptyset$.
\end{enumerate}
\end{proposition}

\begin{remark}
Since $\mathcal D_r\subseteq \mathcal D$, a consequence of Proposition~\ref{T:RKP} is that the stopping region $\mathcal D$ is 
disconnected. In particular, at any time close enough to the pinning date, there exist multiple too-good-to-persist points as well as a stop-loss point.
\end{remark}

\begin{proof}
\begin{enumerate}
\item  [(i)]
Let $0<z_0<z_1<r$ be given. Since $\pi(t,z)$ is continuous and increasing in both variables, we can choose $t_0\in (0,1)$ 
so that $\pi(t,z)\geq \pi(t_0,z_0)> (r+z_1)/(2r)$ for all $(t,z)\in (t_0,1)\times(z_0,z_1)$. Then, for such $(t,z)$,  
Proposition~\ref{T:VEC} implies that
\begin{IEEEeqnarray*}{rCl}
v(t,z) \geq   \pi(t,z)  v_r(t,z) - (1-\pi(t,z))r\geq r(2\pi(t,z)-1)>z_1\geq z.
\end{IEEEeqnarray*}
Consequently, $(t_0,1)\times(z_0,z_1)\subseteq\mathcal C$, which finishes the claim.
\item [(ii)]
Clearly, an equivalent claim is that for all $z_0,z_1\in\R$ with $-r<z_0<z_1<0$ and $t_0\in (0,1)$, the intersection $[t_0, 1) \times (z_0, z_1) \cap \mathcal{D} \neq \emptyset$. Thus assume for a contradiction that there exists $z_0$, $z_1$ and $t_0$ as above such that $[t_0, 1) \times (z_0, z_1) \subseteq \mathcal{C}$, and let $z:=(z_0+z_1)/2$ and $t\in[t_0,1)$.
Let $\hat{Z}$ be a Brownian bridge with $\hat{Z}_1=r$ and $\check{Z}$ a Brownian bridge with $\check{Z}_1=-r$. Then 
$\check Z^{t,z}$ will leave the rectangle $[t_0, 1) \times (z_0, z_1)$ before the pinning time a.s. Define 
\[\tau_i:=\inf\{s\geq 0:\check Z^{t,z}_{t+s}=z_i\},\]
$i=0,1$, and let 
$\eta:=\eta(t,z) := \P (\tau_1<\tau_0)$. Then, introducing a Brownian motion 
$U_s:=z+as+W_s$ with constant drift $a:=-(z_0+r)/(1-t)$,
it is straightforward to check that the pathwise comparison $\check Z_s\leq U_s$ holds on the random time interval
$[0,\tau_0\wedge \tau^U_1]$, where $\tau^U_i=\inf\{s\geq 0:U_s=z_i\}$, $i=0,1$.
Consequently, 
\begin{IEEEeqnarray*}{rCl}
\eta(t,z) &\leq& \P(\tau^U_1<\tau^U_0)\\
&\leq& \P (\sup_{0\leq s < \infty} \frac{-(z_0+r)}{1-t}s +W_s  \geq z_1-z) = \exp( \frac{-2(z_1-z)(z_0+r)}{1-t})
\end{IEEEeqnarray*}
(see p.~251 of \cite{BS02}), so $\eta(t,z)\to 0$ as $t\nearrow 1$.
By comparing with a problem in which the true pinning point is revealed at time
$\tau_0\wedge\tau_1$,
\begin{IEEEeqnarray*}{rCl}
v(t,z) &\leq&  \pi(t,z) v_r(t,z_1) + (1-\pi(t,z))(z_0 + \eta(t,z)(z_1-z_0))  \\
&\leq& \pi(t,z) v_r(t,z_1) + z_0 + \eta(t,z)(z_1-z_0).
\end{IEEEeqnarray*}
Moreover, since $\pi(t,z)\to 0$ as $t\to 1$ by \eqref{E:pi} and $v_r(\cdot,z_1)$ is bounded on $[0,1]$ by \eqref{E:known},
this implies $\limsup_{t\nearrow 1}v(t,z)\leq z_0<z$, which is a contradiction.
Consequently, $[t_0, 1) \times (z_0, z_1)$ cannot have empty intersection with the stopping region.
\end{enumerate}
\end{proof}

\begin{figure}[h]
\centering
\scalebox{0.9}{
\begin{overpic}[scale=0.6, tics=5, ]{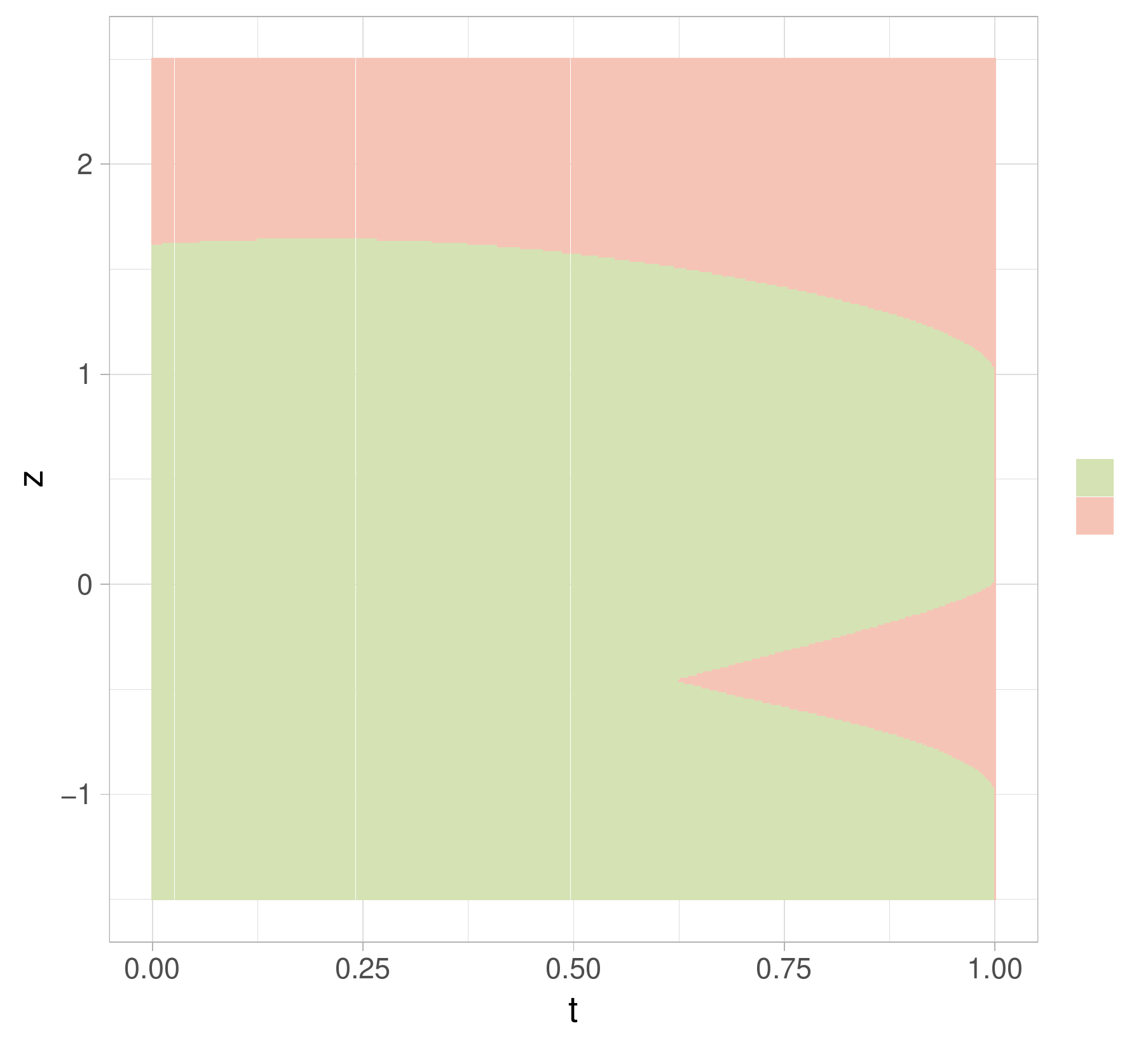}
\put(98,47.8){$\mathcal{C}$}
\put(98,44.3){$\mathcal{D}$}
\end{overpic}}
\caption{Numerically computed $\mathcal{C}$ and $\mathcal{D}$ in the two-point prior case with $p=1/2$, $r=1$.}
\end{figure}

%

\section{Identifying cases with no stop-loss points: the normal and the Gaussian mixture priors}
\label{S:normal}

In this last section, we investigate the cases of the normal and, more generally, the 
Gaussian mixture priors.

\subsection{Normal prior}
In the case $\mu=\mathcal{N}(m, \gamma^2)$, where, $m \in \R$, $\gamma\in (0,1)$, the expression
\begin{IEEEeqnarray*}{rCl}
f(t,y)-\frac{y}{1+t} = \frac{\frac{m}{\gamma^2}+y}{\frac{1}{\gamma^2}+t}- \frac{y}{1+t}
\end{IEEEeqnarray*}
(with $f$ defined in \eqref{E:deff}) is decreasing in $y$. Hence Theorem \ref{T:DY} (\ref{Tp:DDec}) yields existence of $b:[0,1) \to \R$ such that 
$\mathcal{C} = \{ (t,z) \in [0,1) \times [- \infty, \infty] \,:\, z < b(t) \}$. Moreover, in this case $b$ 
is continuous and  decreasing as the next proposition shows.

\begin{proposition}
Let $\mu=\mathcal{N}(m, \gamma^2)$, where, $m \in \R$ and $\gamma\in (0,1)$. Then there exists a decreasing continuous 
$b:[0,1]\to\R$ with 
$b(1)=\frac{m}{1-\gamma^2}$ 
such that $\mathcal{C}=\{(t,z) \in [0,1)\times\R\,:\, z<b(t)\}$. 
Moreover, $b$ is the unique continuous solution of the integral equation
\begin{IEEEeqnarray}{rCl}\label{E:boundary}
b(t)=h(t,b(t))-\int_0^{1-t}\E\left[\frac{h(t+u, Z^{t,b(t)}_{t+u})-Z^{t,b(t)}_{t+u}}{1-t-u}\Ind_{\{Z^{t,b(t)}_{t+u}>b(t+u)\}} \right]\ud u
\end{IEEEeqnarray}
satisfying 
$b(t)\geq \frac{m}{1-\gamma^2}$ for all $t\in[0,1]$.
\end{proposition}

\begin{proof}
The existence of $b:[0,1)\to[-\infty,\infty]$ follows from Theorem~\ref{T:DY} above. 

If $\mu=\mathcal{N}(m, \gamma^2)$ then the drift term in \eqref{E:ZMark} is given by 
$\frac{h(t,z)-z}{1-t}= \frac{m-z(1-\gamma^2)}{1-t(1-\gamma^2)}$, so
\begin{IEEEeqnarray}{rCl} \label{E:ZPL}
\vd Z_t = \frac{\frac{m}{1-\gamma^2}- Z_t}{\frac{1}{1-\gamma^2}-t} \ud t + \ud \hat W_t
\end{IEEEeqnarray}
for $0\leq t<1$. Note that $Z$ in \eqref{E:ZPL} can be viewed as the restriction to the time interval $[0,1]$ of a Brownian bridge 
pinning at $m/(1-\gamma^2)$ at time $T:=\frac{1}{1-\gamma^2}$. 
Thus the problem of stopping a Brownian bridge with a normally distributed pinning point reduces to the problem
of stopping a Brownian bridge with a known pinning point $(T,mT)$, but where stopping is restricted to the time interval $[0,1]$.

Let us define $U_s:= e^s (Z_{T(1-e^{-2s})}-m/(1-\gamma^2))$, $s\geq0$ (this transformation was also used in \cite{EW}). Then
\begin{IEEEeqnarray*}{rCl}
\vd U_s = -U_s \ud s + \sqrt{2T} \ud W_s,
\end{IEEEeqnarray*}
where $W$ is a Brownian motion, i.e.~$U$ is an Ornstein-Uhlenbeck process. Moreover,
\begin{IEEEeqnarray*}{rCl}
V=\sup_{\tau \in\mathcal{T}^Z_{1}}\E[Z_\tau] = \sup_{\tau \in\mathcal{T}^U_{T'}}\E[e^{-s} U_\tau]+m/(1-\gamma^2),
\end{IEEEeqnarray*}
where $T'= \frac{1}{2}\log(\frac{T}{T-1})$.  A Markovian value function $w$ in terms of the process $U$ is defined as
\begin{IEEEeqnarray*}{rCl} \label{E:OSOU}
w(s,u) = \sup_{\tau \in\mathcal{T}^U_{T'-s}}\E[e^{-\tau} U^{s,u}_{s+\tau}], \quad (s,u)\in [0, T']\times\R,
\end{IEEEeqnarray*}
and due to the time-homogeneity of the process $U$, 
the function $w$ is decreasing in the time variable $s$. Consequently, the boundary $g:[0,T'] \to \R$,
the first passage time of $U$ above which is optimal, has to be decreasing. Moreover, all points $(s,u)$ with $u<0$ are automatically
in the continuation region since the operator on the pay-off function is positive at such points, compare \cite[Chapter 10]{bO}. Thus
$g(t)\geq 0>-\infty$. Furthermore, by comparison with the corresponding infinite horizon problem (see \cite{EW}), 
we also have $g(t)<\infty$, so the boundary is finite at all times.

As the problem \eqref{E:OSOU} is in the class of well-understood optimal stopping problems with monotone
boundaries, the continuity of the boundary $g$ and its limiting value $g(T^\prime)$ can be determined using standard arguments (see, e.g.,~\cite[Section 25.2]{PS06}). We omit the details but point out that the limiting value $g(T^\prime)=0$ is the unique zero of the 
corresponding differential operator $\frac{1}{2}\partial_u^2-u\partial_u-1$ acting on the payoff $u$.
Hence, since 
\begin{IEEEeqnarray*}{rCl}
w(s,u)=v(T(1-e^{-2s}), e^{-s}u)-m/(1-\gamma^2), 
\end{IEEEeqnarray*}
there exists a decreasing continuous boundary $b:[0,1] \to \R$ satisfying
\begin{IEEEeqnarray*}{rCl}
b(t)=\sqrt{1-t/T}g(\frac{1}{2}\log \frac{T}{T-t}) + m/(1-\gamma^2)
\end{IEEEeqnarray*}
such that $\mathcal{D}= \{(t,z)\,:\, z\geq b(t) \}$. Since $g(T')=0$, we find that $b(1)=m/(1-\gamma^2)$.

Finally, the fact that $b$ is the unique solution of \eqref{E:boundary} within the class
of continuous functions also follows standard arguments (compare \cite[Chapter 25]{PS06}) and is therefore not included.
\end{proof}

\begin{remark}
In the case $\gamma=1$, the process $Z$ becomes Brownian motion on $[0,1]$, so the problem becomes trivial (by optional sampling, 
any stopping time $\tau$ is optimal). For  $\gamma>1$, the condition of Theorem \ref{T:DY} (\ref{Tp:DInc}) holds and grants existence of a single lower boundary the first passage time below which is optimal. In this case, at the pinning time, the process $Z$ is even more dispersed than a Brownian motion, possibly making the case unnatural to model additional information available about the terminal value. We do not study this case further. 
\end{remark}

\subsection{Gaussian mixtures}
Another natural class of prior distributions to consider is Gaussian mixtures, i.e. distributions of the form
\begin{IEEEeqnarray*}{rCl}
\mu = \sum_{i=1}^n p_i N(m_i, \gamma_i^2), \text{where } n\in \N, \text{ each } m_i \in \R,\, \gamma_i >0,\, p_i \in (0,1), \text{ and } \sum_{i=1}^n p_i = 1.
\end{IEEEeqnarray*}
A useful property of Gaussian mixtures is that they are conjugate priors, i.e.~every possible posterior $\mu_{t,z}$ ($(t,z)\in[0,1]\times \R$) is also a Gaussian mixture parametrised by the same number of parameters as the prior. In particular, $f(t,y)$ in \eqref{E:FF} can be calculated explicitly, making the condition of Theorem \ref{T:DY} checkable analytically. 

\begin{remark} 
One can check that if the components of a Gaussian mixture have the same standard deviations, then a Brownian bridge with such a prior can be viewed as a Brownian bridge pinning at a later time according to a discrete prior. 
\end{remark}

\begin{figure}[!h]
\begin{subfigure}[b]{0.45\textwidth}
\scalebox{0.8}{
\begin{overpic}[scale=0.5, tics=5,trim=0 0 0 0, clip=true]{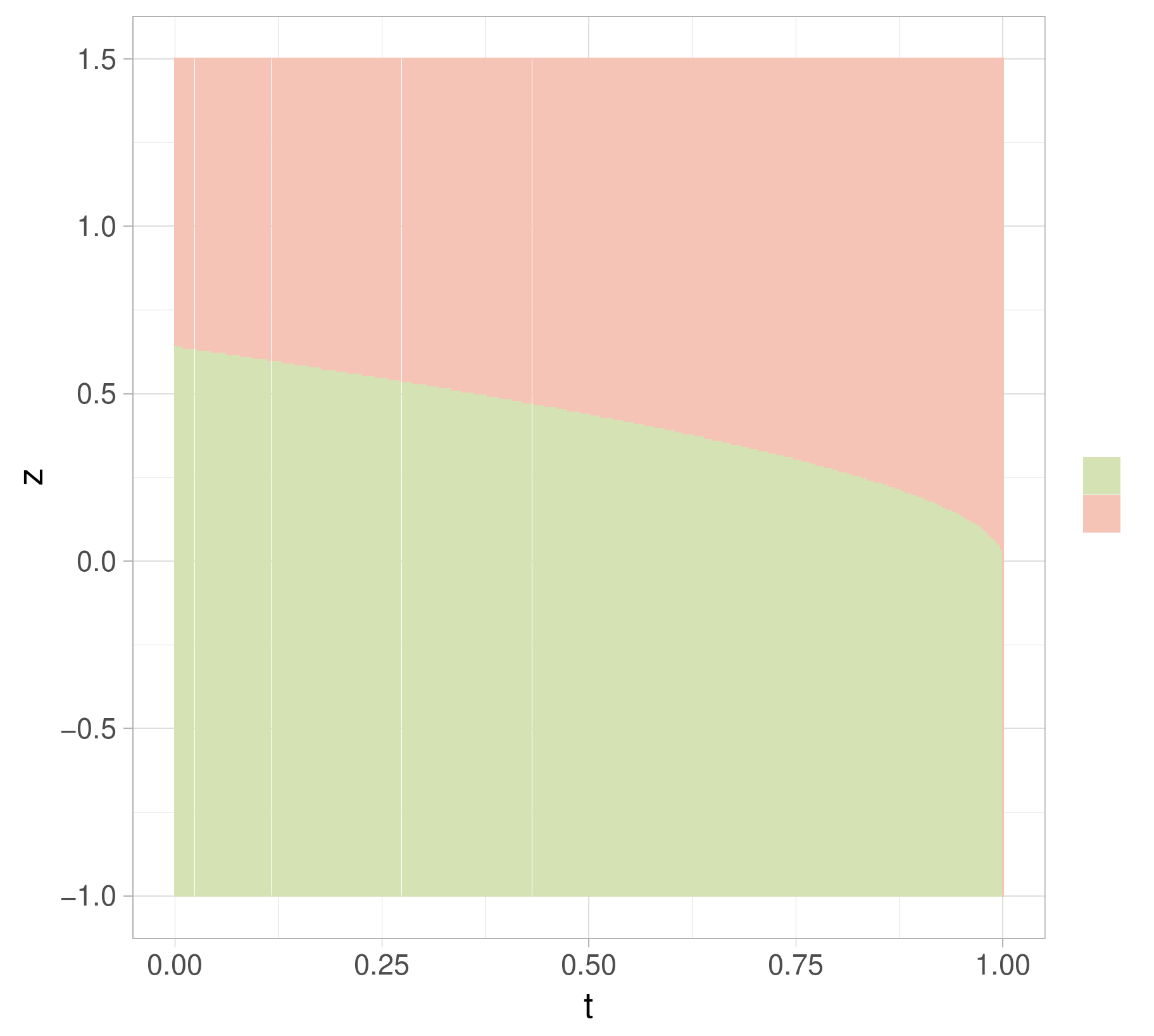}
\put(98, 47.4){\scriptsize $\mathcal{C}$}
\put(98, 44){\scriptsize $\mathcal{D}$}
\end{overpic}}
\scalebox{0.52}{
\begin{overpic}[scale=1, tics=5,trim=0 0 0 0, clip=true]{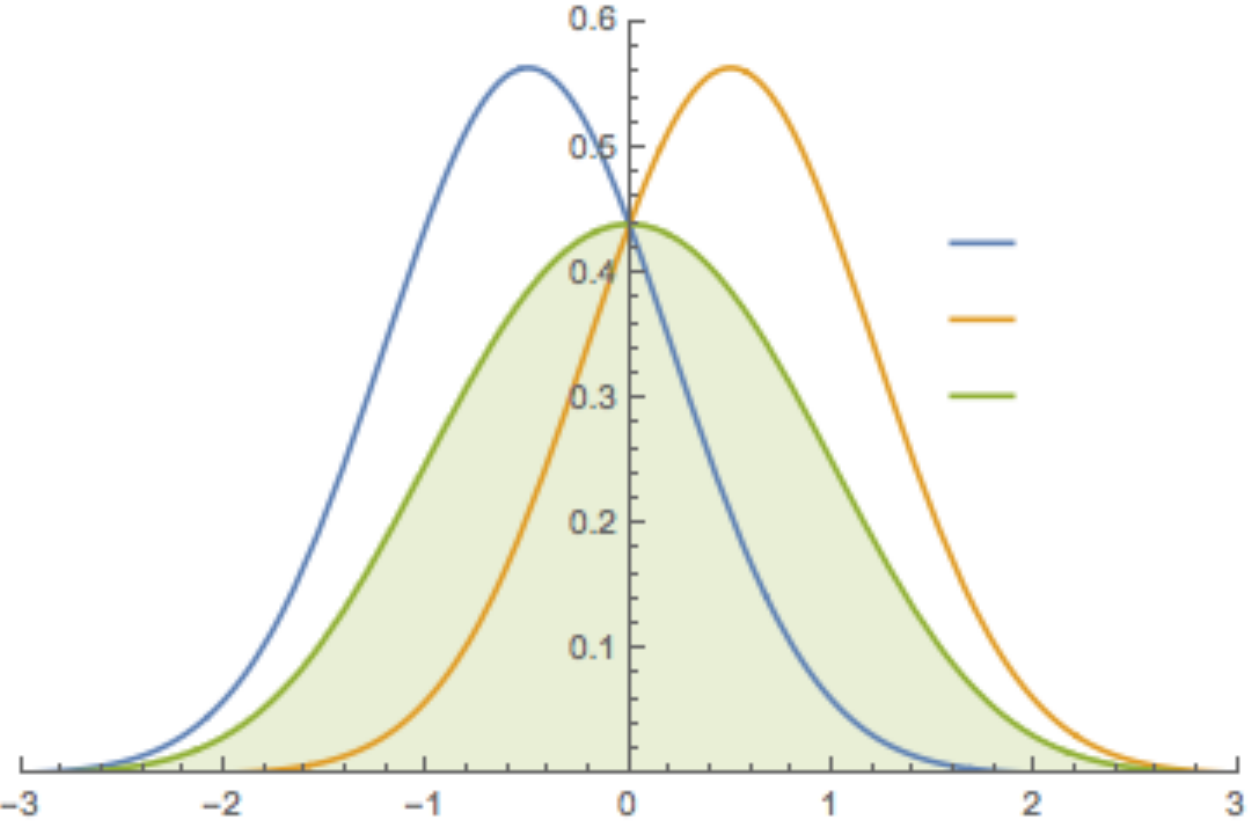}
\put(83, 46.4){\footnotesize $N(-\frac{1}{2},\frac{1}{2})$}
\put(83, 40.4){\footnotesize $N(\frac{1}{2},\frac{1}{2})$}
\put(83, 34.3){\footnotesize $\frac{1}{2} N(-\frac{1}{2},\frac{1}{2})+\frac{1}{2}N(\frac{1}{2},\frac{1}{2})$}
\end{overpic}}
\caption{$r=-l=1/2$, $\gamma^2=1/2$, $p=1/2$}
\label{F:SymGM}
\end{subfigure} \qquad
\begin{subfigure}[b]{0.45\textwidth}
\scalebox{0.8}{
\begin{overpic}[scale=0.5, tics=5]{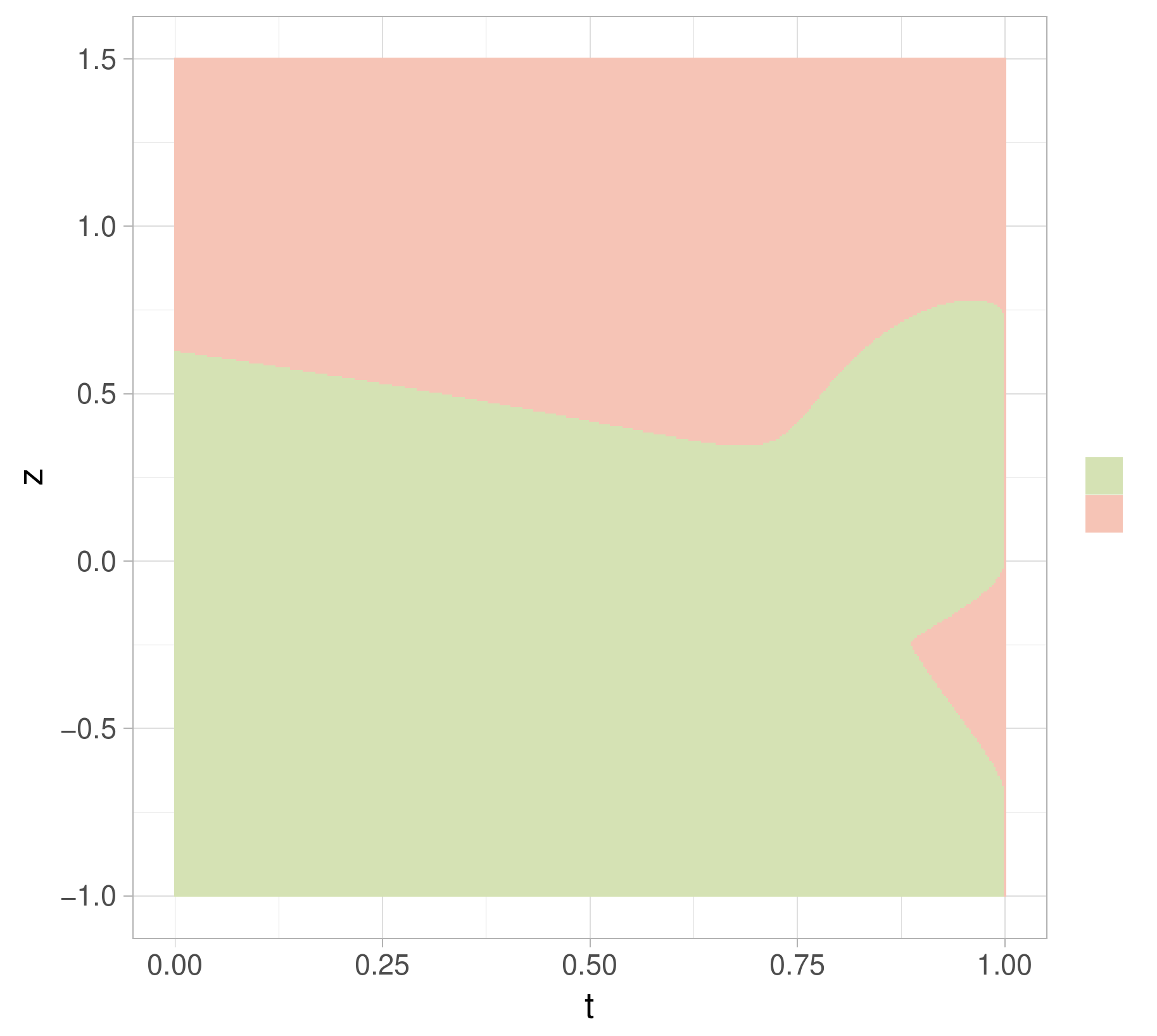}
\put(98, 47.4){\scriptsize $\mathcal{C}$}
\put(98, 44){\scriptsize $\mathcal{D}$}
\end{overpic}}
\scalebox{0.52}{
\begin{overpic}[scale=1, tics=5]{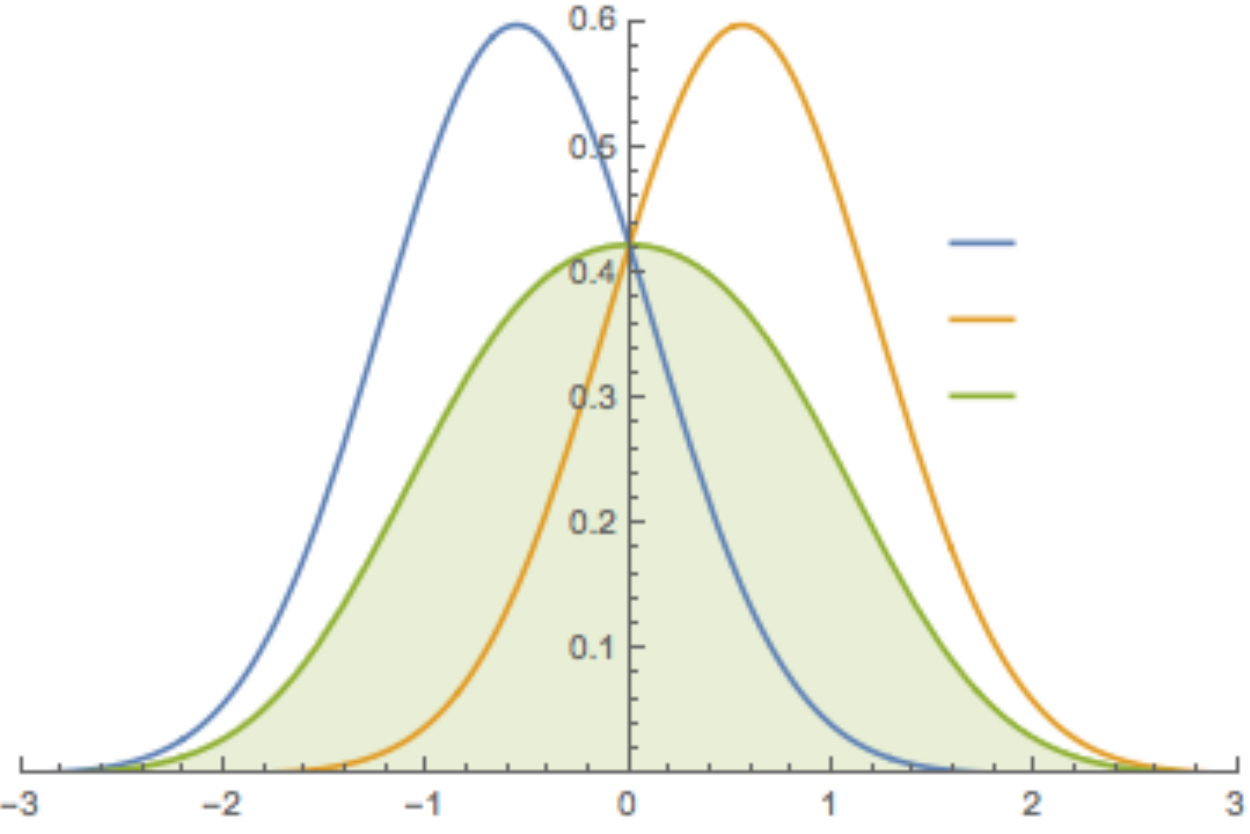}
\put(83, 46.4){\footnotesize $N(-\frac{5}{9},\frac{4}{9})$}
\put(83, 40.4){\footnotesize $N(\frac{5}{9},\frac{4}{9})$}
\put(83, 34.3){\footnotesize $\frac{1}{2} N(-\frac{5}{9},\frac{4}{9})+\frac{1}{2}N(\frac{5}{9},\frac{4}{9})$}
\end{overpic}}
\caption{$r=-l=5/9$, $\gamma^2=4/9$, $p=1/2$} \label{F:GMSymW}
%
%
%
\end{subfigure}
\caption{Numerically computed $\mathcal{C}$ and $\mathcal{D}$ for two similar symmetric Gaussian mixtures whose densities together with the constituent parts are plotted under their respective region plots. In the left plot, the single upper boundary strategy is guaranteed by Theorem~\ref{T:DY}.}
\label{F:GM}
\end{figure}
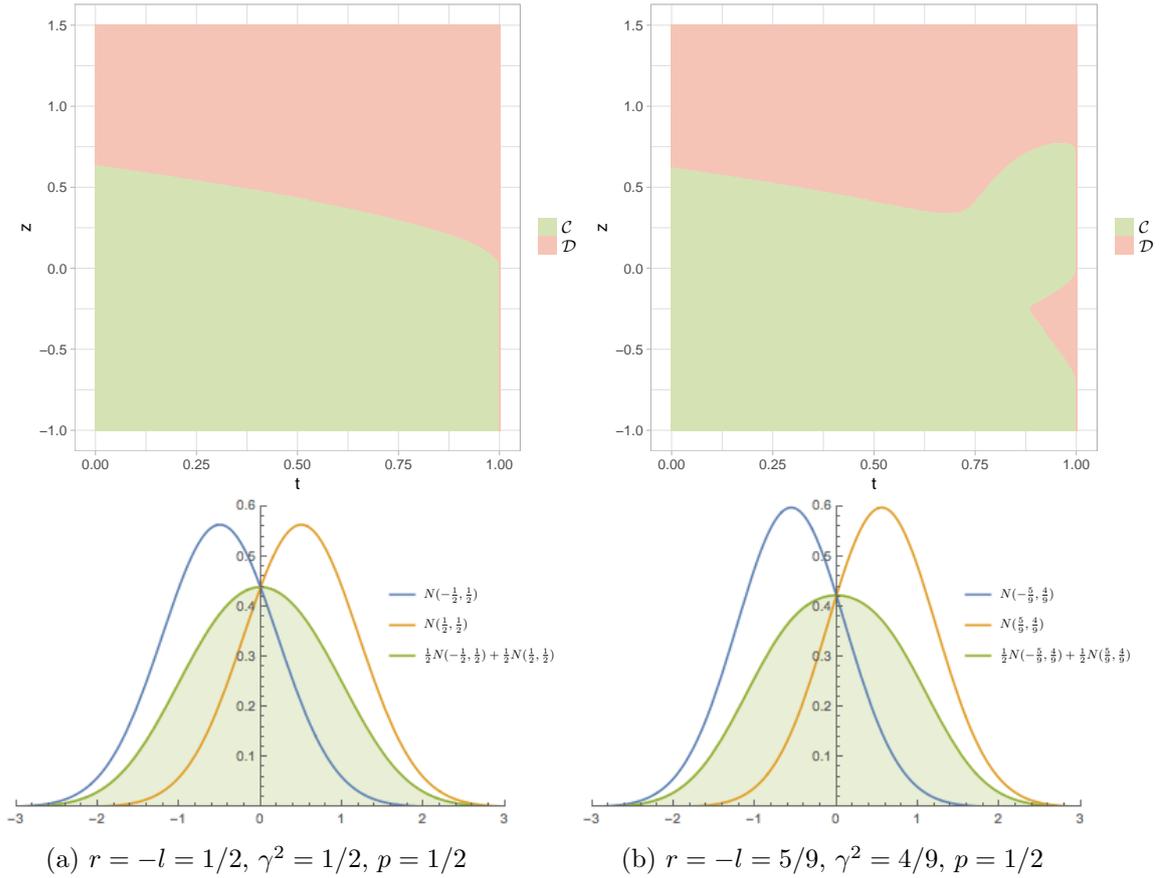
\begin{example}[Two-component Gaussian mixture] Let $\mu = p N(r, \gamma^2) + (1-p) N(l, \eta^2)$, where $r, l \in \R$, $\gamma, \eta>0$, and $p \in (0,1)$.
Calculating the function $f$ explicitly, we obtain
{ 
\begin{IEEEeqnarray*}{rCl} 
\IEEEeqnarraymulticol{3}{l}{
f(t,y)}\\
&=& 
\frac{\frac{1-p}{\gamma}\left(\gamma^2 t+1\right)^{\frac{3}{2}} \left(\eta^2 y+l\right) e^{\frac{1}{2} \left(\frac{r^2}{\gamma^2}+\frac{\left(\eta^2 y+l\right)^2}{\eta^4 t+\eta^2}\right)}
+\frac{p}{\eta}\left(\eta^2 t+1\right)^{\frac{3}{2}} \left(\gamma^2 y+r\right) e^{\frac{1}{2} \left(\frac{\left(\gamma^2 y+r\right)^2}{\gamma^4 t+\gamma^2}+\frac{l^2}{\eta^2}\right)}}{\frac{1-p}{\gamma}\left(\gamma^2 t+1\right)^{\frac{3}{2}} \left(\eta^2 t+1\right) e^{\frac{1}{2} \left(\frac{r^2}{\gamma^2}+\frac{\left(\eta^2 y+l\right)^2}{\eta^4 t+\eta^2}\right)}+ \frac{p}{\eta} \left(\gamma^2 t+1\right) \left(\eta^2 t+1\right)^{\frac{3}{2}} e^{ \frac{1}{2} \left(\frac{\left(\gamma^2 y+r\right)^2}{\gamma^4 t+\gamma^2}+\frac{l^2}{\eta^2}\right)}}.
\end{IEEEeqnarray*}}
As a result, $\frac{\pt}{\pt y} \lt(f(t,y)- y/(1+t) \rt)$ is also an explicit expression, so checking the condition of Theorem \ref{T:DY}  is a matter of examining analytically whether
\begin{IEEEeqnarray}{rCl} \label{E:DerCond}
\frac{\pt}{\pt y} \lt(f(t,y)- y/(1+t) \rt) \leq 0, \quad \text{for all } (t,y).
\end{IEEEeqnarray}
In some interesting special cases, the condition \eqref{E:DerCond} reduces to an easily checkable condition. 
For example, for Gaussian mixtures symmetric around zero, i.e.~$l=-r$, $p=1/2$, $\gamma=\eta$, the condition \eqref{E:DerCond} is satisfied if and only if $0 \leq r \leq \gamma \sqrt{1-\gamma^2}$. Hence the single-upper stopping boundary strategy in
Figure~\ref{F:SymGM} is guaranteed by Theorem~\ref{T:DY}. 
In addition, Figure~\ref{F:GM} illustrates the subtle and sensitive dependence of the stopping strategy on the prior distribution; two seemingly similar symmetric unimodal Gaussian mixture priors yield very different optimal stopping strategies. This illustrates
the high complexity of deriving structural properties of the problem under a general prior.

\end{example}

%
\bibliographystyle{plain}

\end{document}